\documentclass[12pt,a4paper]{article}
\usepackage{fullpage}
\usepackage[utf8]{inputenc}
\usepackage[czech,english]{babel}

\usepackage{lmodern}

\usepackage{amsmath}        % extensions for typesetting of math
\usepackage{amsfonts}       % math fonts
\usepackage{amssymb}
\usepackage{amsthm}         % theorems, definitions, etc.
\usepackage{bbding}         % various symbols (squares, asterisks, scissors, ...)
\usepackage{bm}             % boldface symbols (\bm)
\usepackage{graphicx}       % embedding of pictures
\usepackage{subcaption}
\usepackage{fancyvrb}       % improved verbatim environment
\usepackage[nottoc]{tocbibind} % makes sure that bibliography and the lists
\usepackage{dcolumn}        % improved alignment of table columns
\usepackage{booktabs}       % improved horizontal lines in tables
\usepackage{paralist}       % improved enumerate and itemize
\usepackage{xcolor}         % typesetting in color
\usepackage{tikz-cd}

\usepackage{enumitem}
\usepackage{placeins}

\allowdisplaybreaks
\newcounter{cptTh}
\newtheorem{theorem}[cptTh]{Theorem}

\newtheorem{corollary}[cptTh]{Corollary}
\newtheorem{conjecture}[cptTh]{Conjecture}

\newtheorem{definition}[cptTh]{Definition}
\newtheorem{observation}[cptTh]{Observation}

\usepackage[linesnumbered,longend,noline,figure]{algorithm2e}
\SetKwProg{Fn}{Function}{}{end}
\DontPrintSemicolon

\DefineVerbatimEnvironment{code}{Verbatim}{fontsize=\small, frame=single}

\newcommand{\R}{\mathbb{R}}
\newcommand{\N}{\mathbb{N}}

\newcommand{\abs}[1]{\left|{#1}\right|}

\usepackage{newunicodechar}

\newunicodechar{⟶}{\ensuremath{\longrightarrow}}
\newunicodechar{⟵}{\ensuremath{\longleftarrow}}
\newunicodechar{⟷}{\ensuremath{\longleftrightarrow}}

\newunicodechar{⟹}{\ensuremath{\Longrightarrow}}
\newunicodechar{⟸}{\ensuremath{\Longleftarrow}}
\newunicodechar{⟺}{\ensuremath{\Longleftrightarrow}}

\newunicodechar{→}{\ensuremath{\rightarrow}}
\newunicodechar{←}{\ensuremath{\leftarrow}}
\newunicodechar{↔}{\ensuremath{\leftrightarrow}}

\newunicodechar{⇐}{\ensuremath{\Rightarrow}}
\newunicodechar{⇒}{\ensuremath{\Leftarrow}}
\newunicodechar{⇔}{\ensuremath{\Leftrightarrow}}

\makeatletter

\def\utfch@rundefined{}
\def\defutfmathch@r#1#2{%
  \ifx#2\utfch@rundefined%
  \DeclareUnicodeCharacter{#1}{?%
  }%
  \else%
  \DeclareUnicodeCharacter{#1}{\ensuremath{#2}}%
  \fi%
}

\defutfmathch@r{0391}{A}
\defutfmathch@r{0392}{B}
\defutfmathch@r{0393}{\Gamma}
\defutfmathch@r{0394}{\Delta}
\defutfmathch@r{0395}{E}
\defutfmathch@r{0396}{Z}
\defutfmathch@r{0397}{H}
\defutfmathch@r{0398}{\Theta}
\defutfmathch@r{0399}{I}
\defutfmathch@r{039A}{K}
\defutfmathch@r{039B}{\Lambda}
\defutfmathch@r{039C}{M}
\defutfmathch@r{039D}{N}
\defutfmathch@r{039E}{\Xi}
\defutfmathch@r{039F}{O}
\defutfmathch@r{03A0}{\Pi}
\defutfmathch@r{03A1}{P}
\defutfmathch@r{03A3}{\Sigma}
\defutfmathch@r{03A4}{T}
\defutfmathch@r{03A5}{Y}
\defutfmathch@r{03A6}{\Phi}
\defutfmathch@r{03A7}{X}
\defutfmathch@r{03A8}{\Psi}
\defutfmathch@r{03A9}{\Omega}

\defutfmathch@r{03B1}{\alpha}
\defutfmathch@r{03B2}{\beta}
\defutfmathch@r{03B3}{\gamma}
\defutfmathch@r{03B4}{\delta}
\defutfmathch@r{03B5}{\epsilon}
\defutfmathch@r{03B6}{\zeta}
\defutfmathch@r{03B7}{\eta}
\defutfmathch@r{03B8}{\theta}
\defutfmathch@r{03B9}{\iota}
\defutfmathch@r{03BA}{\kappa}
\defutfmathch@r{03BB}{\lambda}
\defutfmathch@r{03BC}{\mu}
\defutfmathch@r{03BD}{\nu}
\defutfmathch@r{03BE}{\xi}
\defutfmathch@r{03BF}{o}
\defutfmathch@r{03C0}{\pi}
\defutfmathch@r{03C1}{\rho}
\defutfmathch@r{03C2}{\varsigma}
\defutfmathch@r{03C3}{\sigma}
\defutfmathch@r{03C4}{\tau}
\defutfmathch@r{03C5}{\upsilon}
\defutfmathch@r{03C6}{\phi}
\defutfmathch@r{03C7}{\chi}
\defutfmathch@r{03C8}{\psi}
\defutfmathch@r{03C9}{\omega}
\defutfmathch@r{03D1}{\vartheta}
\defutfmathch@r{03D6}{\varpi}

\DeclareUnicodeCharacter{00D7}{\ensuremath{\times}}

\defutfmathch@r{2192}{\to}

\defutfmathch@r{2200}{\forall} % FOR ALL (∀)
\defutfmathch@r{2201}{\complement} % COMPLEMENT (∁)
\defutfmathch@r{2202}{\partial} % PARTIAL DIFFERENTIAL (∂)
\defutfmathch@r{2203}{\exists} % THERE EXISTS (∃)
\defutfmathch@r{2204}{\nexists} % THERE DOES NOT EXIST (∄)
\defutfmathch@r{2205}{\emptyset\varnothing} % EMPTY SET (∅)
\defutfmathch@r{2206}{\Delta} % INCREMENT (∆)
\defutfmathch@r{2207}{\nabla} % NABLA (∇)
\defutfmathch@r{2208}{\in} % ELEMENT OF (∈)
\defutfmathch@r{2209}{\not\in} % NOT AN ELEMENT OF (∉)
\defutfmathch@r{220A}{{\scriptstyle\in}} % SMALL ELEMENT OF (∊)
\defutfmathch@r{220B}{\ni} % CONTAINS AS MEMBER (∋)
\defutfmathch@r{220C}{\not\ni} % DOES NOT CONTAIN AS MEMBER (∌)
\defutfmathch@r{220D}{{\scriptstyle\ni}} % SMALL CONTAINS AS MEMBER (∍)
\defutfmathch@r{220E}{\utfch@rundefined} % END OF PROOF (∎)
\defutfmathch@r{220F}{\prod} % N-ARY PRODUCT (∏)
\defutfmathch@r{2210}{\coprod} % N-ARY COPRODUCT (∐)
\defutfmathch@r{2211}{\sum} % N-ARY SUMMATION (∑)
\defutfmathch@r{2212}{-} % MINUS SIGN (−)
\defutfmathch@r{2213}{\mp} % MINUS-OR-PLUS SIGN (∓)
\defutfmathch@r{2214}{\dot{+}} % DOT PLUS (∔)
\defutfmathch@r{2215}{/} % DIVISION SLASH (∕)
\defutfmathch@r{2216}{\backslash} % SET MINUS (∖)
\defutfmathch@r{2217}{\ast} % ASTERISK OPERATOR (∗)
\defutfmathch@r{2218}{\circ} % RING OPERATOR (∘)
\defutfmathch@r{2219}{\cdot} % BULLET OPERATOR (∙)
\defutfmathch@r{221A}{\sqrt{}} % SQUARE ROOT (√)
\defutfmathch@r{221B}{\sqrt[3]{}} % CUBE ROOT (∛)
\defutfmathch@r{221C}{\sqrt[4]{}} % FOURTH ROOT (∜)
\defutfmathch@r{221D}{\propto} % PROPORTIONAL TO (∝)
\defutfmathch@r{221E}{\infty} % INFINITY (∞)
\defutfmathch@r{221F}{\utfch@rundefined} % RIGHT ANGLE (∟)
\defutfmathch@r{2220}{\angle} % ANGLE (∠)
\defutfmathch@r{2221}{\measuredangle} % MEASURED ANGLE (∡)
\defutfmathch@r{2222}{\sphericalangle} % SPHERICAL ANGLE (∢)
\defutfmathch@r{2223}{\mid} % DIVIDES (∣)
\defutfmathch@r{2224}{\nmid} % DOES NOT DIVIDE (∤)
\defutfmathch@r{2225}{\parallel} % PARALLEL TO (∥)
\defutfmathch@r{2226}{\nparallel} % NOT PARALLEL TO (∦)
\defutfmathch@r{2227}{\wedge} % LOGICAL AND (∧)
\defutfmathch@r{2228}{\vee} % LOGICAL OR (∨)
\defutfmathch@r{2229}{\cap} % INTERSECTION (∩)
\defutfmathch@r{222A}{\cup} % UNION (∪)
\defutfmathch@r{222B}{\int} % INTEGRAL (∫)
\defutfmathch@r{222C}{\iint} % DOUBLE INTEGRAL (∬)
\defutfmathch@r{222D}{\iiint} % TRIPLE INTEGRAL (∭)
\defutfmathch@r{222E}{\oint} % CONTOUR INTEGRAL (∮)
\defutfmathch@r{222F}{\oiint} % SURFACE INTEGRAL (∯)
\defutfmathch@r{2230}{\utfch@rundefined} % VOLUME INTEGRAL (∰)
\defutfmathch@r{2231}{\utfch@rundefined} % CLOCKWISE INTEGRAL (∱)
\defutfmathch@r{2232}{\utfch@rundefined} % CLOCKWISE CONTOUR INTEGRAL (∲)
\defutfmathch@r{2233}{\utfch@rundefined} % ANTICLOCKWISE CONTOUR INTEGRAL (∳)
\defutfmathch@r{2234}{\therefore} % THEREFORE (∴)
\defutfmathch@r{2235}{\because} % BECAUSE (∵)
\defutfmathch@r{2236}{:} % RATIO (∶)
\defutfmathch@r{2237}{::} % PROPORTION (∷)
\defutfmathch@r{2238}{\utfch@rundefined} % DOT MINUS (∸)
\defutfmathch@r{2239}{-:} % EXCESS (∹)
\defutfmathch@r{223A}{\utfch@rundefined} % GEOMETRIC PROPORTION (∺)
\defutfmathch@r{223B}{\utfch@rundefined} % HOMOTHETIC (∻)
\defutfmathch@r{223C}{\sim} % TILDE OPERATOR (∼)
\defutfmathch@r{223D}{\backsim} % REVERSED TILDE (∽)
\defutfmathch@r{223E}{\utfch@rundefined} % INVERTED LAZY S (∾)
\defutfmathch@r{223F}{\utfch@rundefined} % SINE WAVE (∿)
\defutfmathch@r{2240}{\wr} % WREATH PRODUCT (≀)
\defutfmathch@r{2241}{\nsim} % NOT TILDE (≁)
\defutfmathch@r{2242}{\utfch@rundefined} % MINUS TILDE (≂)
\defutfmathch@r{2243}{\simeq} % ASYMPTOTICALLY EQUAL TO (≃)
\defutfmathch@r{2244}{\not\simeq} % NOT ASYMPTOTICALLY EQUAL TO (≄)
\defutfmathch@r{2245}{\utfch@rundefined} % APPROXIMATELY EQUAL TO (≅)
\defutfmathch@r{2246}{\utfch@rundefined} % APPROXIMATELY BUT NOT ACTUALLY EQUAL TO (≆)
\defutfmathch@r{2247}{\utfch@rundefined} % NEITHER APPROXIMATELY NOR ACTUALLY EQUAL TO (≇)
\defutfmathch@r{2248}{\approx} % ALMOST EQUAL TO (≈)
\defutfmathch@r{2249}{\not\approx} % NOT ALMOST EQUAL TO (≉)
\defutfmathch@r{224A}{\utfch@rundefined} % ALMOST EQUAL OR EQUAL TO (≊)
\defutfmathch@r{224B}{\utfch@rundefined} % TRIPLE TILDE (≋)
\defutfmathch@r{224C}{\utfch@rundefined} % ALL EQUAL TO (≌)
\defutfmathch@r{224D}{\asymp} % EQUIVALENT TO (≍)
\defutfmathch@r{224E}{\Bumpeq} % GEOMETRICALLY EQUIVALENT TO (≎)
\defutfmathch@r{224F}{\bumpeq} % DIFFERENCE BETWEEN (≏)
\defutfmathch@r{2250}{\doteq} % APPROACHES THE LIMIT (≐)
\defutfmathch@r{2251}{\doteqdot} % GEOMETRICALLY EQUAL TO (≑)
\defutfmathch@r{2252}{\fallingdotseq} % APPROXIMATELY EQUAL TO OR THE IMAGE OF (≒)
\defutfmathch@r{2253}{\risingdotseq} % IMAGE OF OR APPROXIMATELY EQUAL TO (≓)
\defutfmathch@r{2254}{:=} % COLON EQUALS (≔)
\defutfmathch@r{2255}{=:} % EQUALS COLON (≕)
\defutfmathch@r{2256}{\utfch@rundefined} % RING IN EQUAL TO (≖)
\defutfmathch@r{2257}{\circeq} % RING EQUAL TO (≗)
\defutfmathch@r{2258}{\utfch@rundefined} % CORRESPONDS TO (≘)
\defutfmathch@r{2259}{\utfch@rundefined} % ESTIMATES (≙)
\defutfmathch@r{225A}{\utfch@rundefined} % EQUIANGULAR TO (≚)
\defutfmathch@r{225B}{\utfch@rundefined} % STAR EQUALS (≛)
\defutfmathch@r{225C}{\triangleq} % DELTA EQUAL TO (≜)
\defutfmathch@r{225D}{\utfch@rundefined} % EQUAL TO BY DEFINITION (≝)
\defutfmathch@r{225E}{\utfch@rundefined} % MEASURED BY (≞)
\defutfmathch@r{225F}{\utfch@rundefined} % QUESTIONED EQUAL TO (≟)
\defutfmathch@r{2260}{\not=} % NOT EQUAL TO (≠)
\defutfmathch@r{2261}{\equiv} % IDENTICAL TO (≡)
\defutfmathch@r{2262}{\not\equiv} % NOT IDENTICAL TO (≢)
\defutfmathch@r{2263}{\utfch@rundefined} % STRICTLY EQUIVALENT TO (≣)
\defutfmathch@r{2264}{\leq} % LESS-THAN OR EQUAL TO (≤)
\defutfmathch@r{2265}{\geq} % GREATER-THAN OR EQUAL TO (≥)
\defutfmathch@r{2266}{\leqq} % LESS-THAN OVER EQUAL TO (≦)
\defutfmathch@r{2267}{\geqq} % GREATER-THAN OVER EQUAL TO (≧)
\defutfmathch@r{2268}{\gneqq} % LESS-THAN BUT NOT EQUAL TO (≨)
\defutfmathch@r{2269}{\lneqq} % GREATER-THAN BUT NOT EQUAL TO (≩)
\defutfmathch@r{226A}{\ll} % MUCH LESS-THAN (≪)
\defutfmathch@r{226B}{\gg} % MUCH GREATER-THAN (≫)
\defutfmathch@r{226C}{\between} % BETWEEN (≬)
\defutfmathch@r{226D}{\utfch@rundefined} % NOT EQUIVALENT TO (≭)
\defutfmathch@r{226E}{\nless} % NOT LESS-THAN (≮)
\defutfmathch@r{226F}{\ngtr} % NOT GREATER-THAN (≯)
\defutfmathch@r{2270}{\nleq} % NEITHER LESS-THAN NOR EQUAL TO (≰)
\defutfmathch@r{2271}{\ngeq} % NEITHER GREATER-THAN NOR EQUAL TO (≱)
\defutfmathch@r{2272}{\utfch@rundefined} % LESS-THAN OR EQUIVALENT TO (≲)
\defutfmathch@r{2273}{\utfch@rundefined} % GREATER-THAN OR EQUIVALENT TO (≳)
\defutfmathch@r{2274}{\utfch@rundefined} % NEITHER LESS-THAN NOR EQUIVALENT TO (≴)
\defutfmathch@r{2275}{\utfch@rundefined} % NEITHER GREATER-THAN NOR EQUIVALENT TO (≵)
\defutfmathch@r{2276}{\lessgtr} % LESS-THAN OR GREATER-THAN (≶)
\defutfmathch@r{2277}{\gtrless} % GREATER-THAN OR LESS-THAN (≷)
\defutfmathch@r{2278}{\not\lessgtr} % NEITHER LESS-THAN NOR GREATER-THAN (≸)
\defutfmathch@r{2279}{\not\gtrless} % NEITHER GREATER-THAN NOR LESS-THAN (≹)
\defutfmathch@r{227A}{\prec} % PRECEDES (≺)
\defutfmathch@r{227B}{\succ} % SUCCEEDS (≻)
\defutfmathch@r{227C}{\preceq} % PRECEDES OR EQUAL TO (≼)
\defutfmathch@r{227D}{\succeq} % SUCCEEDS OR EQUAL TO (≽)
\defutfmathch@r{227E}{\precsim} % PRECEDES OR EQUIVALENT TO (≾)
\defutfmathch@r{227F}{\succsim} % SUCCEEDS OR EQUIVALENT TO (≿)
\defutfmathch@r{2280}{\not\prec} % DOES NOT PRECEDE (⊀)
\defutfmathch@r{2281}{\not\succ} % DOES NOT SUCCEED (⊁)
\defutfmathch@r{2282}{\subset} % SUBSET OF (⊂)
\defutfmathch@r{2283}{\supset} % SUPERSET OF (⊃)
\defutfmathch@r{2284}{\not\subset} % NOT A SUBSET OF (⊄)
\defutfmathch@r{2285}{\not\supset} % NOT A SUPERSET OF (⊅)
\defutfmathch@r{2286}{\subseteq} % SUBSET OF OR EQUAL TO (⊆)
\defutfmathch@r{2287}{\subseteq} % SUPERSET OF OR EQUAL TO (⊇)
\defutfmathch@r{2288}{\not\subseteq} % NEITHER A SUBSET OF NOR EQUAL TO (⊈)
\defutfmathch@r{2289}{\not\supseteq} % NEITHER A SUPERSET OF NOR EQUAL TO (⊉)
\defutfmathch@r{228A}{\utfch@rundefined} % SUBSET OF WITH NOT EQUAL TO (⊊)
\defutfmathch@r{228B}{\utfch@rundefined} % SUPERSET OF WITH NOT EQUAL TO (⊋)
\defutfmathch@r{228C}{\utfch@rundefined} % MULTISET (⊌)
\defutfmathch@r{228D}{\utfch@rundefined} % MULTISET MULTIPLICATION (⊍)
\defutfmathch@r{228E}{\uplus} % MULTISET UNION (⊎)
\defutfmathch@r{228F}{\sqsubset} % SQUARE IMAGE OF (⊏)
\defutfmathch@r{2290}{\sqsupset} % SQUARE ORIGINAL OF (⊐)
\defutfmathch@r{2291}{\sqsubseteq} % SQUARE IMAGE OF OR EQUAL TO (⊑)
\defutfmathch@r{2292}{\sqsupseteq} % SQUARE ORIGINAL OF OR EQUAL TO (⊒)
\defutfmathch@r{2293}{\sqcap} % SQUARE CAP (⊓)
\defutfmathch@r{2294}{\sqcup} % SQUARE CUP (⊔)
\defutfmathch@r{2295}{\oplus} % CIRCLED PLUS (⊕)
\defutfmathch@r{2296}{\ominus} % CIRCLED MINUS (⊖)
\defutfmathch@r{2297}{\otimes} % CIRCLED TIMES (⊗)
\defutfmathch@r{2298}{\oslash} % CIRCLED DIVISION SLASH (⊘)
\defutfmathch@r{2299}{\odot} % CIRCLED DOT OPERATOR (⊙)
\defutfmathch@r{229A}{\circledcirc} % CIRCLED RING OPERATOR (⊚)
\defutfmathch@r{229B}{\circledast} % CIRCLED ASTERISK OPERATOR (⊛)
\defutfmathch@r{229C}{} % CIRCLED EQUALS (⊜)
\defutfmathch@r{229D}{\circleddash} % CIRCLED DASH (⊝)
\defutfmathch@r{229E}{\boxplus} % SQUARED PLUS (⊞)
\defutfmathch@r{229F}{\boxminus} % SQUARED MINUS (⊟)
\defutfmathch@r{22A0}{\boxtimes} % SQUARED TIMES (⊠)
\defutfmathch@r{22A1}{\boxdot} % SQUARED DOT OPERATOR (⊡)
\defutfmathch@r{22A2}{\vdash} % RIGHT TACK (⊢)
\defutfmathch@r{22A3}{\dashv} % LEFT TACK (⊣)
\defutfmathch@r{22A4}{\top} % DOWN TACK (⊤)
\defutfmathch@r{22A5}{\bot} % UP TACK (⊥)
\defutfmathch@r{22A6}{\vdash} % ASSERTION (⊦)
\defutfmathch@r{22A7}{\models} % MODELS (⊧)
\defutfmathch@r{22A8}{\utfch@rundefined} % TRUE (⊨)
\defutfmathch@r{22A9}{\utfch@rundefined} % FORCES (⊩)
\defutfmathch@r{22AA}{\utfch@rundefined} % TRIPLE VERTICAL BAR RIGHT TURNSTILE (⊪)
\defutfmathch@r{22AB}{\utfch@rundefined} % DOUBLE VERTICAL BAR DOUBLE RIGHT TURNSTILE (⊫)
\defutfmathch@r{22AC}{\not\vdash} % DOES NOT PROVE (⊬)
\defutfmathch@r{22AD}{\utfch@rundefined} % NOT TRUE (⊭)
\defutfmathch@r{22AE}{\utfch@rundefined} % DOES NOT FORCE (⊮)
\defutfmathch@r{22AF}{\utfch@rundefined} % NEGATED DOUBLE VERTICAL BAR DOUBLE RIGHT TURNSTILE (⊯)
\defutfmathch@r{22B0}{\utfch@rundefined} % PRECEDES UNDER RELATION (⊰)
\defutfmathch@r{22B1}{\utfch@rundefined} % SUCCEEDS UNDER RELATION (⊱)
\defutfmathch@r{22B2}{\lhd} % NORMAL SUBGROUP OF (⊲)
\defutfmathch@r{22B3}{\rhd} % CONTAINS AS NORMAL SUBGROUP (⊳)
\defutfmathch@r{22B4}{\unlhd} % NORMAL SUBGROUP OF OR EQUAL TO (⊴)
\defutfmathch@r{22B5}{\unrhd} % CONTAINS AS NORMAL SUBGROUP OR EQUAL TO (⊵)
\defutfmathch@r{22B6}{\utfch@rundefined} % ORIGINAL OF (⊶)
\defutfmathch@r{22B7}{\utfch@rundefined} % IMAGE OF (⊷)
\defutfmathch@r{22B8}{\utfch@rundefined} % MULTIMAP (⊸)
\defutfmathch@r{22B9}{\utfch@rundefined} % HERMITIAN CONJUGATE MATRIX (⊹)
\defutfmathch@r{22BA}{\utfch@rundefined} % INTERCALATE (⊺)
\defutfmathch@r{22BB}{\veebar} % XOR (⊻)
\defutfmathch@r{22BC}{\barwedge} % NAND (⊼)
\defutfmathch@r{22BD}{\utfch@rundefined} % NOR (⊽)
\defutfmathch@r{22BE}{\utfch@rundefined} % RIGHT ANGLE WITH ARC (⊾)
\defutfmathch@r{22BF}{\utfch@rundefined} % RIGHT TRIANGLE (⊿)
\defutfmathch@r{22C0}{\bigwedge} % N-ARY LOGICAL AND (⋀)
\defutfmathch@r{22C1}{\bigvee} % N-ARY LOGICAL OR (⋁)
\defutfmathch@r{22C2}{\bigcap} % N-ARY INTERSECTION (⋂)
\defutfmathch@r{22C3}{\bigcup} % N-ARY UNION (⋃)
\defutfmathch@r{22C4}{\utfch@rundefined} % DIAMOND OPERATOR (⋄)
\defutfmathch@r{22C5}{\cdot} % DOT OPERATOR (⋅)
\defutfmathch@r{22C6}{\utfch@rundefined} % STAR OPERATOR (⋆)
\defutfmathch@r{22C7}{\utfch@rundefined} % DIVISION TIMES (⋇)
\defutfmathch@r{22C8}{\bowtie} % BOWTIE (⋈)
\defutfmathch@r{22C9}{\utfch@rundefined} % LEFT NORMAL FACTOR SEMIDIRECT PRODUCT (⋉)
\defutfmathch@r{22CA}{\utfch@rundefined} % RIGHT NORMAL FACTOR SEMIDIRECT PRODUCT (⋊)
\defutfmathch@r{22CB}{\utfch@rundefined} % LEFT SEMIDIRECT PRODUCT (⋋)
\defutfmathch@r{22CC}{\utfch@rundefined} % RIGHT SEMIDIRECT PRODUCT (⋌)
\defutfmathch@r{22CD}{\utfch@rundefined} % REVERSED TILDE EQUALS (⋍)
\defutfmathch@r{22CE}{\utfch@rundefined} % CURLY LOGICAL OR (⋎)
\defutfmathch@r{22CF}{\utfch@rundefined} % CURLY LOGICAL AND (⋏)
\defutfmathch@r{22D0}{\Subset} % DOUBLE SUBSET (⋐)
\defutfmathch@r{22D1}{\Supset} % DOUBLE SUPERSET (⋑)
\defutfmathch@r{22D2}{\Cap} % DOUBLE INTERSECTION (⋒)
\defutfmathch@r{22D3}{\Cup} % DOUBLE UNION (⋓)
\defutfmathch@r{22D4}{\pitchfork} % PITCHFORK (⋔)
\defutfmathch@r{22D5}{\utfch@rundefined} % EQUAL AND PARALLEL TO (⋕)
\defutfmathch@r{22D6}{\lessdot} % LESS-THAN WITH DOT (⋖)
\defutfmathch@r{22D7}{\gtrdot} % GREATER-THAN WITH DOT (⋗)
\defutfmathch@r{22D8}{\lll} % VERY MUCH LESS-THAN (⋘)
\defutfmathch@r{22D9}{\ggg} % VERY MUCH GREATER-THAN (⋙)
\defutfmathch@r{22DA}{\lesseqgtr} % LESS-THAN EQUAL TO OR GREATER-THAN (⋚)
\defutfmathch@r{22DB}{\gtreqless} % GREATER-THAN EQUAL TO OR LESS-THAN (⋛)
\defutfmathch@r{22DC}{\utfch@rundefined} % EQUAL TO OR LESS-THAN (⋜)
\defutfmathch@r{22DD}{\utfch@rundefined} % EQUAL TO OR GREATER-THAN (⋝)
\defutfmathch@r{22DE}{\utfch@rundefined} % EQUAL TO OR PRECEDES (⋞)
\defutfmathch@r{22DF}{\utfch@rundefined} % EQUAL TO OR SUCCEEDS (⋟)
\defutfmathch@r{22E0}{\utfch@rundefined} % DOES NOT PRECEDE OR EQUAL (⋠)
\defutfmathch@r{22E1}{\utfch@rundefined} % DOES NOT SUCCEED OR EQUAL (⋡)
\defutfmathch@r{22E2}{\utfch@rundefined} % NOT SQUARE IMAGE OF OR EQUAL TO (⋢)
\defutfmathch@r{22E3}{\utfch@rundefined} % NOT SQUARE ORIGINAL OF OR EQUAL TO (⋣)
\defutfmathch@r{22E4}{\utfch@rundefined} % SQUARE IMAGE OF OR NOT EQUAL TO (⋤)
\defutfmathch@r{22E5}{\utfch@rundefined} % SQUARE ORIGINAL OF OR NOT EQUAL TO (⋥)
\defutfmathch@r{22E6}{\utfch@rundefined} % LESS-THAN BUT NOT EQUIVALENT TO (⋦)
\defutfmathch@r{22E7}{\utfch@rundefined} % GREATER-THAN BUT NOT EQUIVALENT TO (⋧)
\defutfmathch@r{22E8}{\utfch@rundefined} % PRECEDES BUT NOT EQUIVALENT TO (⋨)
\defutfmathch@r{22E9}{\utfch@rundefined} % SUCCEEDS BUT NOT EQUIVALENT TO (⋩)
\defutfmathch@r{22EA}{\not\lhd} % NOT NORMAL SUBGROUP OF (⋪)
\defutfmathch@r{22EB}{\utfch@rundefined} % DOES NOT CONTAIN AS NORMAL SUBGROUP (⋫)
\defutfmathch@r{22EC}{\utfch@rundefined} % NOT NORMAL SUBGROUP OF OR EQUAL TO (⋬)
\defutfmathch@r{22ED}{\utfch@rundefined} % DOES NOT CONTAIN AS NORMAL SUBGROUP OR EQUAL (⋭)
\defutfmathch@r{22EE}{\vdots} % VERTICAL ELLIPSIS (⋮)
\defutfmathch@r{22EF}{\cdots} % MIDLINE HORIZONTAL ELLIPSIS (⋯)
\defutfmathch@r{22F0}{\utfch@rundefined} % UP RIGHT DIAGONAL ELLIPSIS (⋰)
\defutfmathch@r{22F1}{\ddots} % DOWN RIGHT DIAGONAL ELLIPSIS (⋱)
\defutfmathch@r{22F2}{\utfch@rundefined} % ELEMENT OF WITH LONG HORIZONTAL STROKE (⋲)
\defutfmathch@r{22F3}{\utfch@rundefined} % ELEMENT OF WITH VERTICAL BAR AT END OF HORIZONTAL STROKE (⋳)
\defutfmathch@r{22F4}{\utfch@rundefined} % SMALL ELEMENT OF WITH VERTICAL BAR AT END OF HORIZONTAL STROKE (⋴)
\defutfmathch@r{22F5}{\utfch@rundefined} % ELEMENT OF WITH DOT ABOVE (⋵)
\defutfmathch@r{22F6}{\utfch@rundefined} % ELEMENT OF WITH OVERBAR (⋶)
\defutfmathch@r{22F7}{\utfch@rundefined} % SMALL ELEMENT OF WITH OVERBAR (⋷)
\defutfmathch@r{22F8}{\utfch@rundefined} % ELEMENT OF WITH UNDERBAR (⋸)
\defutfmathch@r{22F9}{\utfch@rundefined} % ELEMENT OF WITH TWO HORIZONTAL STROKES (⋹)
\defutfmathch@r{22FA}{\utfch@rundefined} % CONTAINS WITH LONG HORIZONTAL STROKE (⋺)
\defutfmathch@r{22FB}{\utfch@rundefined} % CONTAINS WITH VERTICAL BAR AT END OF HORIZONTAL STROKE (⋻)
\defutfmathch@r{22FC}{\utfch@rundefined} % SMALL CONTAINS WITH VERTICAL BAR AT END OF HORIZONTAL STROKE (⋼)
\defutfmathch@r{22FD}{\utfch@rundefined} % CONTAINS WITH OVERBAR (⋽)
\defutfmathch@r{22FE}{\utfch@rundefined} % SMALL CONTAINS WITH OVERBAR (⋾)
\defutfmathch@r{22FF}{\utfch@rundefined} % Z NOTATION BAG MEMBERSHIP (⋿)
\makeatother

\def\Z{{\mathbb Z}}

\def\R{{\mathbb R}}

\def\set#1{\left\{#1\right\}}

\def\abs#1{\left|#1\right|}

\def\ceil#1{\left\lceil#1\right\rceil}

\newcommand{\eg}{e.g\., \ignorespaces}
\newcommand{\ie}{i.e\., \ignorespaces}

\newcounter{todo}
\makeatletter

\newcommand\listtodoname{TODO seznam}
\newcommand\listoftodos{%
  \ifnum\thetodo=0\else%
  \section*{\listtodoname}\@starttoc{tod}%
  \addcontentsline{toc}{section}{\listtodoname}%
  \fi}%

\usepackage[listformat=simple]{caption}

\def\ext@table{lof}
\let\c@table\c@figure
\let\ftype@table\ftype@figure

\usepackage{calc}

\def\doi#1{doi: \href{http://dx.doi.org/#1}{\nolinkurl{#1}}}

\def\Caption#1#2{\caption[#1]{#1 #2}}

\def\tabname{\kern -0.5pt T\kern-1.2pt a\kern -0.5pt b\.\kern-1pt}
\DeclareCaptionListFormat{fig}{\hbox to 0pt{\if\IsAlg TAlg\else Fig\fi\.\hss}\hphantom{\tabname} #2}
\DeclareCaptionListFormat{tab}{\hbox to 0pt{\tabname\hss}\hphantom{\tabname} #2}

\def\IsAlg{F}
\let\realalgorithm\algorithm
\def\algorithm{
  \def\IsAlg{T}%
  \def\figurename{Algorithm}%
  \realalgorithm}

\let\realtable\table
\def\table{\captionsetup{listformat=tab}\realtable}

\let\realfigure\figure
\def\figure{\captionsetup{listformat=fig}\realfigure}

\def\l@figure{\@dottedtocline{1}{0em}{4.2em}}
\let\l@table\l@figure

\makeatother

\let\realfootnoterule\footnoterule
\def\footnoterule{\vfill\realfootnoterule}

\usepackage{pgfplots}
\usepackage[colorlinks=true]{hyperref}
\def\.{\mbox{.}}

\def\Boundary#1#2{\ensuremath{\left( #1 \,|\, #2 \right\}}}
\def\J{\ensuremath{\mathcal J}}

\def\1{{\bf 1}}

\def\Par{\nu}
\def\kgadgets#1{{\ensuremath{\mathcal G^{#1}}}}

\def\Experiment#1{%
  \href{https://gitlab.kam.mff.cuni.cz/radek/cdc-counting/-/blob/master/experiments/#1}{%
  \nolinkurl{#1}}}

\hypersetup{pdftitle=Counting Circuit Double Covers}
\hypersetup{pdfauthor={Radek Hušek, Robert Šámal}}

\begin{document}
\title{Counting Circuit Double Covers}
\author{
Radek Hušek\thanks{
\texttt{radek.husek@fit.cvut.cz}, Faculty of Information Technology, Czech Technical University in Prague,
Czech Republic
} \and Robert Šámal\thanks{
\texttt{samal@iuuk.mff.cuni.cz},
Computer Science Institute of Charles University,
Faculty of Mathematics and Physics, Charles University,
Prague, Czech Republic
}
}
\date{}
\maketitle

\begin{abstract}
We study a counting version of Cycle Double Cover Conjecture. We discuss
why it is more interesting to count circuits (\ie graphs isomorphic to
$C_k$ for some~$k$) instead of cycles (graphs with all degrees even).
We give an almost-exponential lower-bound for graphs with a surface
embedding of representativity at least 4.
We also prove an exponential lower-bound
for planar graphs. We conjecture that any bridgeless cubic graph
has at least $2^{n/2 - 1}$ circuit double covers and we show an infinite
class of graphs for which this bound is tight.

\end{abstract}

%We consider only cubic graphs and we allow multiple edges (note that
%loops are not interesting as they imply existence of a bridge)
%unless explicitly noted otherwise.
Several recent results and conjectures study counting versions of classical
existence statements.
Esperet et al\. \cite{EKKKN-expmatch} proved Lov\'asz--Plummer conjecture:
\begin{theorem}[Esperet et al\.~\cite{EKKKN-expmatch}]
Every bridgeless cubic graph has exponentially many perfect matchings.
\end{theorem}

A similar result for colorings of planar graphs was proven by Thomassen~\cite{Thom-exp5-col}:

\begin{theorem}[Thomassen \cite{Thom-exp5-col}]
Every planar graph has exponentially many (list) 5-vertex-colorings.
\end{theorem}

Thomassen~\cite{Thom-exp3} also conjectured existence of exponentially many 3-vertex-colorings
of triangle-free planar graphs. He gave a subexponential bound that was
later improved by Asadi et al.~\cite{ADPT-exp3}.
However, the conjecture stays open.
By duality, these results and conjecture can be equivalently stated for the
number of nowhere-zero
$\Z_5$-flows (or $\Z_3$-flows) of planar (4-edge-connected) graphs.
Dvořák, Mohar and the second author~\cite{DMS19} extended this to non-planar graphs:

\begin{theorem}[Dvořák, Mohar and Šámal~\cite{DMS19}]
Every 3-edge-connected graph has exponentially many $\Z_2 \times \Z_3$-flows, 
every 4-edge-connected graph has exponentially many $\Z_2 \times \Z_2$-flows and 
every 6-edge-connected graph has exponentially many $\Z_3$-flows. 
\end{theorem}

We ask the same question for circuit double covers of cubic graphs
(see Definitions~\ref{def:circuit} and~\ref{def:cdc} below). This is
motivated by Cycle Double Cover conjecture:

\begin{conjecture}[Szekeres '73 \cite{sze73}, Seymour '79 \cite{sey79}]\label{con:cdc}
  Every bridgeless graph has a cycle double cover.
\end{conjecture}

We give an exponential
bound for planar graphs (Theorem~\ref{thm:planar-exp-lb}),
almost exponential bound for graphs on any other
fixed surface (Corollary~\ref{cor:surface-lb}) and we present a strengthening
of Conjecture~\ref{con:cdc} which is motivated by our results in
Section~\ref{sec:lin-rep}
(note that Corollary~\ref{cor:triangle} shows this conjecture is tight for
Klee graphs):
\begin{conjecture}\label{con:cdc-count}
  Every bridgeless cubic graph with $n$ vertices has at least $2^{n/2 - 1}$ circuit
  double covers.
\end{conjecture}
% reference only Cor not previous Obs

The structure of this paper is the following: 
We first discuss why we chose circuit double covers (or CiDCs for short)
over cycle double covers (CyDCs for short).
Then we show a construction that gives many CiDCs for graphs with a surface
embedding of representativity at least 4.
This gives an almost exponential bound but we observe that this
technique cannot be applied to all graph sequences -- in particular it cannot
be applied to Flower snarks.

In Section~\ref{sec:lin-rep} we present a very condensed version of the linear representation
of circuit double covers developed in the PhD thesis of the first author~\cite{radek-phd}
and prove enough about it to make this paper self-contained.
Using this framework we prove that planar graphs have exponentially many CiDCs.

%%%%%%%%%%%%%%%%%%%%%%%%%%%%%%%%%%%%%%%%%%%%%%%%%%%%%%%%%%%%%%%%%%%%%%%%%

\section{Circuit vs\. Cycle}

The existential questions usually ask for a cycle double cover (and more
strongly for 5-CyDC). Obviously a CyDC exists if and only if a CiDC exists
and asking for CyDC gives an opportunity to restrict the number of cycles
used.
The main reason
we count CiDCs instead is that a single CiDC
corresponds to multiple CyDCs and this number heavily depends on
the allowed number of cycles. This is nicely illustrated by
Observation~\ref{obs:cdc_blowup}.

\begin{definition}[Circuit and Cycle]\label{def:circuit}
  A {\em circuit} is a connected 2-regular subgraph
  (\ie a subgraph isomorphic to $C_k$ for some~$k$).
  A {\em cycle} is a subgraph with all degrees even (\ie a union
  of edge-disjoint circuits).
\end{definition}

\begin{definition}[Double Cover]\label{def:cdc}
  Let $G$ be a graph.
  A multiset of circuits (cycles, respectively) $\mathcal C$ is
  a {\em circuit} ({\em cycle}, respectively)
  \em{double cover} if every edge of $G$ is contained in exactly two elements of
  $\mathcal C$.
  It is a $k$-cycle double cover if $\abs{\mathcal C} \leq k$.
  We denote $\nu(G)$ the number of circuit double covers of the graph $G$.
\end{definition}

\begin{observation}\label{obs:cdc_blowup}
  The circuit double cover given by an embedding of a graph into the plane
  can be made into $2^{\Omega(n_3)}$ many 5-cycle double covers and
  $2^{\Omega(n_3 \log n_3)}$ cycle double covers just by grouping circuits into
  cycles in different ways where $n_3$ is number of vertices of degree at least 3.
\end{observation}
\begin{proof}
  Any graph with a vertex of degree 1 does not have any CyDC and subdividing an edge
  (\ie addition of a vertex of degree 2) does not change the number of CyDC so we may assume
  that the graph has minimum degree at least 3. Hence the number of faces is~$\Omega(n)$.

  The faces of bridgeless planar graph can be 4-colored.
  Some color contains at least one quarter of the faces, so 
  if we let $k$ be the number of faces of the largest color class, then 
  $k = \Omega(n)$.
  For a 5-CyDC we split this color class into two
  cycles in $2^{k-1}$ ways (we fix one circuit to be able to distinguish
  the cycles), hence there are $2^{\Omega(n)}$ such 5-CyDCs.
  For a general CyDC we split this color class into any number of cycles which
  can be done in $B_k$ ways.\footnote{
    The Bell number $B_k$ is the number of possible partitions of a set of size
    $k$. The name was introduced by Becker and Riordan~\cite{bell-numbers}.
%    It is known that $B_k = 2^{\Theta(k \log k)}$.
  }
  As a simple lower bound we can take half of the color class to create cycles
  and divide the rest of the color class among them in an arbitrary fashion
  leading to $\Omega(n)^{\Omega(n)} = 2^{\Omega(n \log n)}$ CyDCs.
  (More precise estimates of $B_k$ are known but they are not relevant here.)
\end{proof}

\begin{figure}[tb]
  \centering
  \includegraphics[width=.35\textwidth]{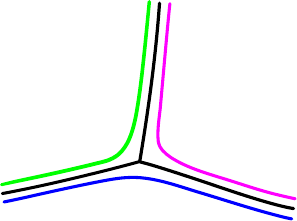}
  \caption{CiDC around a vertex of degree 3}
  \label{fig:cdc-3-vertex}
\end{figure}

\begin{figure}[tb]
  \centering
  \includegraphics[width=.95\textwidth]{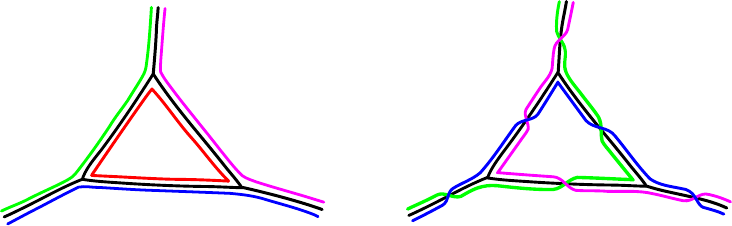}
  \caption{Two possible CiDCs of a triangle gadget}
  \label{fig:cdc-triangle}
\end{figure}

%%%%%%%%%%%%%%%%%%%%%%%%%%%%%%%%%%%%%%%%%%%%%%%%%%%%%%%%%%%%%%%%%%%%%%%%%%%%%%%%%

\section{Representation of Circuit Double Covers}\label{sec:cdc-rep}

Circuit double covers are usually represented as (multi)sets of circuits.
Although this is a natural representation, it has several downsides.
The most important is that there is no simple way to do local modifications
in this representation. For an alternative representation we use the fact
that for cubic graphs
any CiDC around any vertex looks the same (Figure~\ref{fig:cdc-3-vertex}).
Because
all the vertices look the same, the only thing that can change is how
are the walks connected to each other at the edges.

There are two ways to join the walks at each edge. To be able to distinguish
these two ways we assume that the graph is drawn into a plane, and then
the two ways correspond to crossing or not crossing the walks at the edge.
This is shown in Figure~\ref{fig:cdc-ex1}. We require that all the vertices are
distinct points, edges do not pass through the vertices and do not touch
but of course they can cross each other. This drawing serves us as a way to fix
the rotation system (the order of edges around each vertex).
(Alternatively we can view each such configuration as a {\em ribbon graph}
-- see, \eg Ellis-Monaghan and Moffatt~\cite{graphs-on-surfaces-2} section 1.1.4.)
Every CiDC of a cubic graph can be represented uniquely in
this way. This leads to the following simple but important observation:

\begin{figure}[tb]
  \centering
  \includegraphics[width=.4\textwidth]{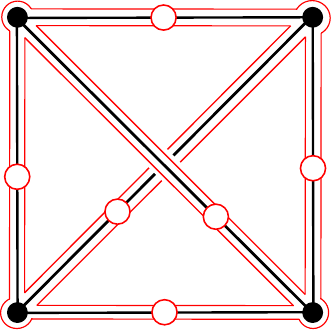}
  \caption{A drawing of $K_4$}
  \label{fig:cdc-ex1}
\end{figure}

\begin{observation}
  A cubic graph with $n$ vertices has at most $2^{3n/2}$ circuit double covers.
\end{observation}
\begin{proof}
  A cubic graphs has $3n/2$ edges, there are two ways how to join the walks at every
  edge and every CiDC can be described in this way.
\end{proof}

Another consequence of this representation is the following:
When we take a closed disk for each circuit in the cover 
and identify the disk boundary with the edges of the graph, we get a closed surface. 
Thus, counting CiDCs of a connected cubic graph~$G$ is the same as counting embeddings of~$G$ where every face boundary 
is a circuit. 

Some of our proofs will require cyclically 4-edge-connected graphs.
We recall the definition of cyclic connectivity and a useful lemma about it.
We sometimes omit the word ``edge'' as we do not use vertex connectivity
in this paper.

\begin{definition}[Cyclic Connectivity]\label{def:cyclic-connectivity}
  A graph is {\em cyclically $k$-edge-connected} if every circuit-separating
  edge cut has size at least $k$.
\end{definition}

\begin{observation}\label{obs:cyc-rem2}
  Let $G$ be a cyclically 4-edge-connected cubic graph. Let $G'$ be
  a graph created from $G$ by removing two non-adjacent edges.
  Then $G'$ is 2-edge-connected.
\end{observation}
\begin{proof}
  Because $G$ is cyclically 4-edge-connected it is also 3-edge-connected
  (any cut smaller than three in a cubic graph is non-trivial).
  Hence $G'$ is connected. It remains to show that it is bridgeless.

  We proceed by contradiction. Let $b$ be a bridge of $G' = G - \set{e, f}$.
  Because $b$ is a cut in $G'$, the set $C = \set{b, e, f}$
  is a 3-cut in $G$. The graph $G$ is cyclically 4-edge-connected so $C$
  must be a trivial cut (\ie edges around one vertex). This is a contradiction
  with $e$ and $f$ being non-adjacent.
\end{proof}

We observe that graphs with 2-cuts or non-trivial 3-cuts
can be reduced to smaller graphs:

\begin{observation}[2-cut]\label{obs:cdc_2cut}
  Let $G$ be a cubic graph with a cut of size 2 and denote $G_1$ and $G_2$
  graphs obtained by contracting one side of the cut and suppressing
  the new vertex of degree 2.
  Then $\Par(G) = 2 \Par(G_1) \Par(G_2)$.
\end{observation}
\begin{proof}
  Edges of a 2-cut are covered by the same circuits in any CiDC.
  Let $e$ (resp\. $f$) be the edge of $G_1$ (resp\. $G_2$) created
  by contraction of the other side of the cut and suppression of
  the created 2-vertex.

  First we show the inequality $\geq$:
  Fix any CiDC of $G_1$ and $G_2$.
  Denote $w_{11}$ and $w_{12}$ the circuits covering $e$
  and $w_{21}$ and $w_{22}$ the circuits covering $f$.
  We can construct $G$ from $G_1$ and $G_2$ by cutting $e$ and $f$ and
  joining each half of $e$ with a half of $f$. We can do the same
  for the CiDCs. We have a choice whether to join $w_{11}$ with $w_{21}$
  or $w_{22}$. It is easy to observe that both choices give a valid
  CiDC of $G$ and these CiDCs are different (because two circuits in a CiDC
  of a cubic graph are never the same). Hence $\Par(G) \geq 2 \Par(G_1) \Par(G_2)$.

  For the opposite inequality, observe that every CiDC of $G$ is uniquely described
  by a CiDC of $G_1$, a CiDC of $G_2$ and the way in which the circuits covering
  $e$ a $f$ are joined.
\end{proof}

\begin{observation}[3-cut]\label{obs:cdc_3cut}
  Let $G$ be a graph with cut of size 3 and denote $G_1$ and $G_2$
  graphs obtained by contracting one side of the cut.
  Then $\Par(G) = \Par(G_1) \Par(G_2)$.
\end{observation}
\begin{proof}
  Let $e, f, g$ be the edges of a 3-cut. In any CiDC of $G$ this 3-cut is covered
  by three circuits $C_1, C_2, C_3$ such that $e$ is covered by
  $C_1, C_2$, $f$ is covered by $C_2, C_3$ and $g$ is covered by
  $C_1, C_3$. Similarly to the previous observation, a CiDC of $G$ is
  uniquely determined by CiDCs of $G_1$ and $G_2$. There is no choice in
  this case as circuit covering $e$ and $f$ in $G_1$ must be joined
  with the circuit covering $e$ and $f$ in $G_2$ -- all other choices
  lead to a closed walk which uses some edge twice.
\end{proof}

%%%%%%%%%%%%%%%%%%%%%%%%%%%%%%%%%%%%%%%%%%%%%%%%%%%%%%%%%%%%%%%%%%%%%%%%%

\section{The Flower Construction}

First we present a construction which starts with a graph embedded into some surface
and creates almost exponentially many circuit double covers by doing local
changes to the CiDC given by the embedding.
Despite its name, this construction is not related to Flower snarks in any way.
The basic idea is to choose a face $f$
and modify the CiDC on $f$ and its neighbors (let $n_f$ denote the number of
its neighbor faces)
so we get $2^{\Omega(n_f)}$ CiDC.

This is shown in Figure~\ref{fig:flower_basic} -- every
small circle there denotes a choice whether to cross the walks or not, leading to
exponentially many possibilities. 
We will show below in Observation~\ref{obs:flower_cdc_count} that a constant fraction of the choices leads to 
a CiDC with three circuits in the flower. 
(Other choices lead to one or two self-touching walks.)
Note that when we redo the CiDC on the flower, we cover the outer edges only
once as they are once covered from the outside and we do not want to modify
that part of the CiDC. We model this by taking the flower as a standalone planar
graph embedded in such a way that the rest of the original graph would be
in the outer face and counting all CiDCs which fix the walk corresponding to the
outer face. We call such objects {\em outer-fixed CiDCs} (but we will
sometimes omit ``outer-fixed'' as for the flowers we consider only the
outer-fixed CiDCs).

\begin{definition}[Flower]\label{def:flower}
  Let $G$ be an embedding of a graph into some surface. We say that a face $f$ of size $k$
  together
  with its neighbor faces are a {\em $k$-flower with center $f$} if the following
  is satisfied:
  \begin{enumerate}[noitemsep,topsep=0pt,parsep=0pt,partopsep=0pt]
    \item boundaries of $f$ and all its neighbour faces are cycles,
    \item $k \geq 3$,
    \item $f$ shares exactly one edge with each of its neighbor faces,
    \item consecutive neighbor faces also share exactly one edge, and
    \item non-consecutive neighbor faces do not share any edge.
  \end{enumerate}
\end{definition}

\begin{observation}\label{obs:flower-center}
  Let $G$ be an embedding of a cyclically 4-edge-connected cubic graph
  into a plane.
  Then every face of $G$ is a centre of a flower.
\end{observation}
\begin{proof}
Let $f$ be a $k$-face of $G$. Observe that a closed curve which separates
some vertices and crosses the planar embedding only $i$ times,
proves that the graph contains an edge cut of size $i$.
Hence we can show that violating any of the requirements
leads to a small cut which contradicts the cyclical 4-edge-connectivity:
\begin{enumerate}
  \item If a vertex repeats on a boundary so must an edge because $G$ is cubic.
    Let $e$ be the repeated edge. Edge $e$ is a bridge because there exists a closed
    curve separating its endpoints and crossing only $e$ and no other edge or vertex.
  \item If $k = 1$, then the boundary of $f$
    is a loop which in a cubic graph
    implies existence of a bridge.
    If $k = 2$ then $f$ is bounded by two parallel edges which leads
    to a 2-cut.
  \item If $f$ shares two or more edges with some other face $p$ then there is
    a closed curve cutting only two of the edges shared by them. This cut
    cannot be trivial because it would require a vertex of degree two.
  \item We also find a similar curve if two consecutive neighbor faces of $f$
    share more than one edge.
  \item If two nonconsecutive neighbor faces of $f$ -- denote them $p_1$ and $p_2$ --
    share an edge then there is a closed curve passing through faces $f$, $p_1$ and
    $p_2$ and cutting only the three edges that the faces share.
    \qedhere
\end{enumerate}
\end{proof}

We will also need that we can edit each flower to obtain a lot of CiDCs.
We show that the exact number of outer-fixed CiDCs of a $k$-flower
is $\frac{2^{k-1} + (-1)^k}{3} + 1$ in our paper~\cite{lin-rep}.
But the proof requires additional definitions
so we state a good enough lower bound instead:

\begin{figure}[b]
  \centering
  \includegraphics[scale=0.8]{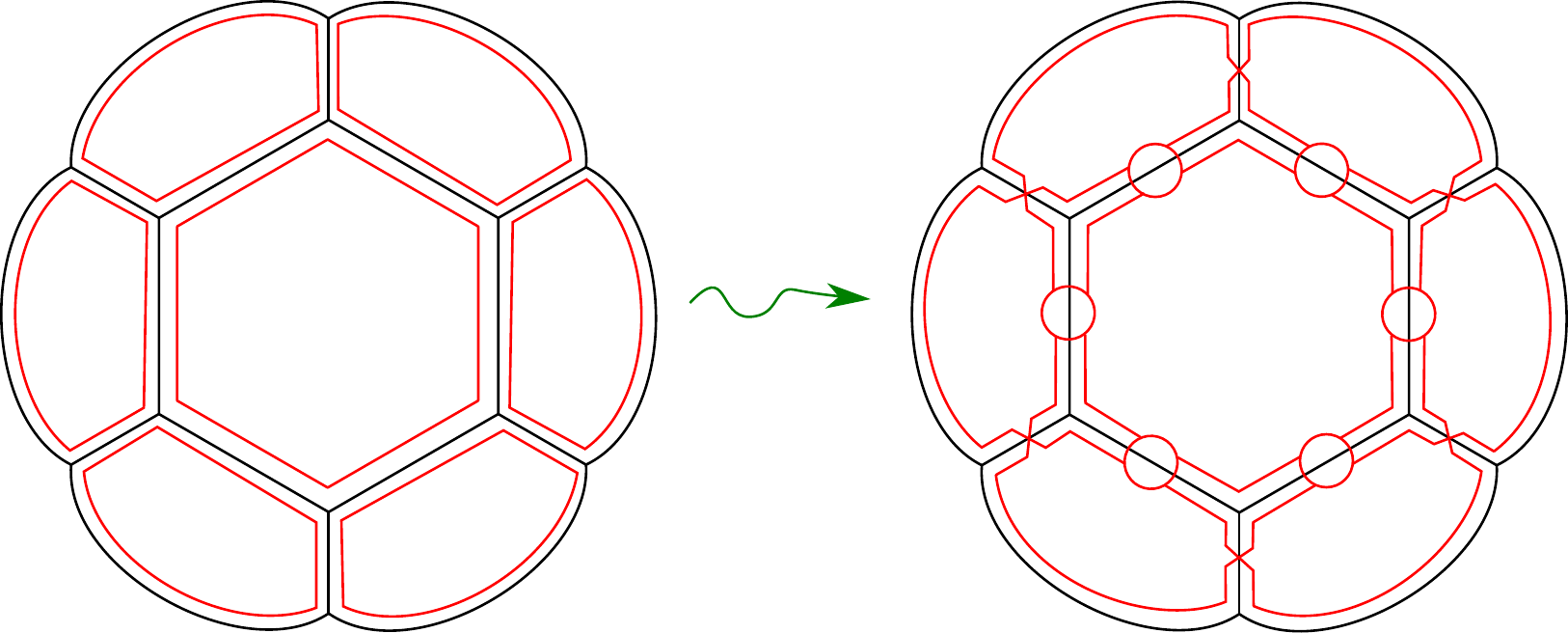}
  \Caption{The basic idea of the flower construction\ignorespaces}{.
    Circles denote the possible choices.\label{fig:flower_basic}}
\end{figure}

\begin{figure}[tb]
  \centering
  \includegraphics[scale=1.5]{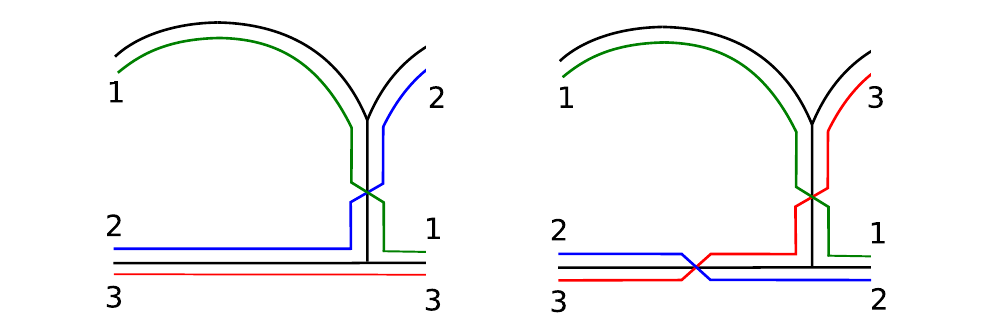}
  \caption[Flower 4-pole]{Flower 4-pole with its two outer-fixed CiDCs we consider.\label{fig:flower_step}}
\end{figure}

\begin{observation}\label{obs:flower_cdc_count}
  The flower of size $k$ has at least $2^{k-3} + 1$ outer-fixed circuit double covers.
\end{observation}
\begin{proof}
The ``$+1$'' is the original CiDC given by the embedding.
Note that this CiDC has $k+1$ circuits but all the new CiDCs we create
have exactly three circuits.

Other CiDCs are consisting of three walks winding around the central
face following the idea in Figure~\ref{fig:flower_basic} -- each flower
is a cycle of pieces (technically 4-poles, see Section~\ref{sec:lin-rep}) 
shown in Figure~\ref{fig:flower_step} with one of the depicted outer-fixed CiDCs.
The only difference between these CiDCs is whether we crossed the walks
on the edge of central face.
Hence we have a binary choice on each edge of the central face.
All possible choices lead to closed walks but some of them might join
two (or all three) of them together. If the walks are joined together,
this walk is self-touching and thus not a circuit.

We show that we can choose arbitrarily all the crossings except for some three
consecutive ones and use these three to make sure that the result is
a CiDC leading to $2^{k-3}$ possible CiDCs.
Figure~\ref{fig:flower_finish} shows three consecutive choice points
and the walks leaving it on the right-hand side are numbered 1 to 3.
We distinguish two cases:
\begin{itemize}
  \item The walk 1 appears on the upper edge at the left-hand side
    (the upper part of the figure): In this case we must not cross
    on the choice point of the middle edge so walk 1 stays on the upper edge
    at the right-hand side.
    We might choose arbitrarily on the left choice point but the choice
    on the right edge is forced so we do not join walks 2 and 3 together.
  \item The walk 1 appears on the lower edge at the left-hand side
    (the lower part of the figure): In this case we must choose on the
    left choice point so walk 1 after this point is below the edge.
    Then we cross in the middle and hence walk 1 ends on the upper
    edge on the right-hand side. And again we flip the right choice-point
    so we do not join walks 2 and 3 together.\qedhere
\end{itemize}
\end{proof}

\begin{figure}[tb]
  \centering
  \includegraphics[scale=1]{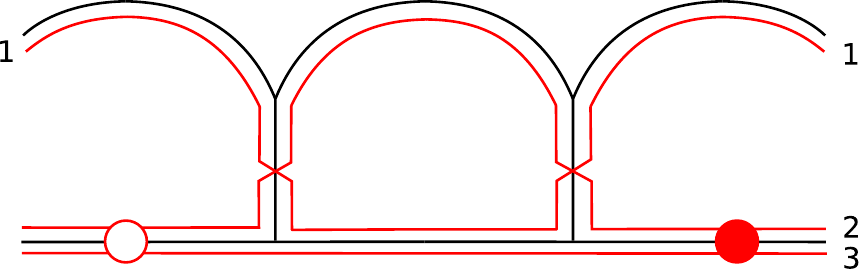}\\\vskip 4ex
  \includegraphics[scale=1]{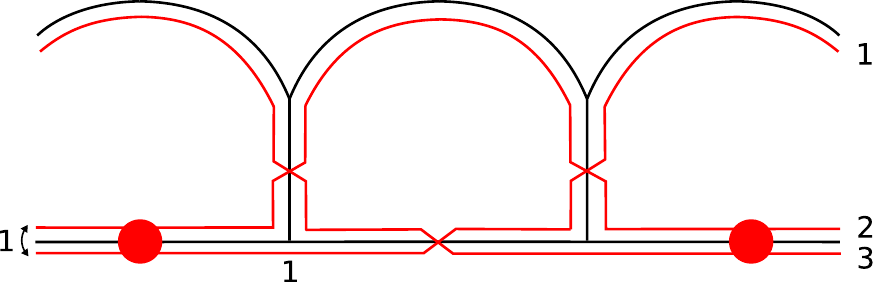}
  \caption[Finishing the flower]{Finishing the flower.
    Full circles denote forced choices.\label{fig:flower_finish}}
\end{figure}

\begin{definition}[Representativity~\cite{RS88}]
  A simple closed curve which does not split the surface into two parts is
  called {\em noncontractible}. Given an embedding of a graph into a surface
  which is not a sphere, the {\em representativity} of this embedding is
  the minimal number of points of the embedding that any closed noncontractible curve
  must intersect.
\end{definition}

\begin{observation}\label{obs:flower-center-gen}
  Let $G$ be an embedding of a cyclically 4-edge-connected cubic graph
  into a surface $\sigma$ with representativity at least 4.
  Then every face of $G$ is a centre of a flower.
\end{observation}
\begin{proof}
The same as proof of Observation~\ref{obs:flower-center} but a closed
curve with $i$ crossings we obtain might not separate any vertices. It
is not contractible in this case and hence it shows that representativity
of the embedding
is at most $i$ instead of the graph having a cut of size $i$.
\end{proof}

Note that in contrast with small cuts we do not know how
to work around small representativity.

\section{Graphs on Surfaces}

Now we use the flower construction to obtain almost exponential lower bounds
for graphs on surfaces.
To choose a suitable set of flowers we need their centers to be at distance 3 or more.
This is well modeled by the vertex coloring of the square of the dual graph.
For general surfaces we approximate $\chi(G^2) \leq \Delta^2 + 1$ but for the
plane there is much better bound:

\begin{theorem}[Molloy and Salavatipour~\cite{molloy2005bound}]\label{thm:planar_square_color}
The chromatic number of the square of a planar graph $G$ with maximum degree $\Delta$
is $\chi(G^2) \leq \ceil{\tfrac{5}{3}\Delta} + 78$.
\end{theorem}

\begin{theorem}\label{thm:flower_planar}
  Every bridgeless cubic planar graph with $n \geq 4$ vertices
  has $2^{\sqrt{n}/100}$ circuit double covers.
\end{theorem}
\begin{proof}
  Let $G$ be the given graph. If $G$ is not cyclically 4-edge-connected
  we apply Observation~\ref{obs:cdc_2cut} or~\ref{obs:cdc_3cut}.
  Let $n_1$ and $n_2$ be the number of vertices of $G_1$ and $G_2$ from
  the observations.
  The application of Observation~\ref{obs:cdc_2cut} gives us:
  $$\nu(G) = 2\nu(G_1)\nu(G_2) \geq 2^{\sqrt{n_1}/100 + \sqrt{n_2}/100 + 1}
  \geq 2^{\sqrt{n_1 + n_2}/100} = 2^{\sqrt{n}/100}
  $$
  where the last inequality holds because $n_1 + n_2 = n$ and
  it holds $\sqrt{x} + \sqrt{y} \geq \sqrt{x + y}$ for all non-negative $x, y$.
  Similarly for Observation~\ref{obs:cdc_3cut} where $n_1 + n_2 = n + 2$.

  If $G$ is cyclically 4-edge-connected, then every face is a center
  of a flower due to Observation~\ref{obs:flower-center}
  and we distinguish two cases depending
  on $\Delta$, the maximum size of a face:
  \begin{itemize}
    \item $\Delta \geq \sqrt{n}$: We
  choose the largest face to be the center of a flower and this flower itself gives
  us $\max\{2^{\sqrt{n}-3}, 2\} \geq 2^{\sqrt{n}/100}$ CiDCs.
    \item $\Delta \leq \sqrt{n}$: We apply
  Theorem~\ref{thm:planar_square_color} to the dual showing that its chromatic
  number is at most $\ceil{\tfrac{5}{3}\Delta} + 78 \leq 100\Delta \leq 100\sqrt{n}$.
  We can choose flowers in such a way
  that there is at least $n/(100\sqrt{n}) = \sqrt{n}/100$ of them. Every
  flower has at least two CiDCs so in total we have at least $2^{\sqrt{n}/100}$ CiDCs.
  \qedhere
  \end{itemize}
\end{proof}

For general surfaces we have to also consider that Euler
characteristic $\chi$ influences the number of faces (note that $\chi$ is
negative except for plane, projective plane, torus and Klein bottle) but otherwise
the proof is similar:

\begin{theorem}\label{thm:flower_general}
  Every bridgeless cubic graph with $n \geq 4$ vertices and with embedding in
  a surface of Euler characteristic $\chi$
  with representativity at least 4 has $2^{\sqrt[3]{n + 2\chi}/100}$
  circuit double covers.
\end{theorem}
\begin{proof}
  The number of the faces is
  $f = \frac{1}{2}(n + 2\chi)$ due to Euler's formula.
  If the graph is not cyclically 4-edge-connected
  we apply Observation~\ref{obs:cdc_2cut} or~\ref{obs:cdc_3cut}. Note that there
  exists a closed curve on the surface which crosses the embedding of the graph
  only in the cut edges such that it is a separating curve (otherwise there is
  a face which is not a disk)
  so we obtain embeddings of the smaller graphs and we can proceed.
  For the number of the faces $f_1$, $f_2$ of the smaller graphs it holds
  $f_1 + f_2 = f + 2$ in the case of a 2-cut and Observation~\ref{obs:cdc_2cut}
  gives us:
  $$\nu(G) = 2\nu(G_1)\nu(G_2) \geq 2^{\sqrt[3]{2f_1}/100 + \sqrt[3]{2f_2}/100 + 1}
  \geq 2^{\sqrt[3]{2f_1 + 2f_2}/100} \geq 2^{\sqrt[3]{2f}/100}
  $$
  and similarly for 3-cut where $f_1 + f_2 = f+3$.
  Hence we can proceed by induction.

  Now let us assume that the graph $G$ is 4-edge-connected. 
  Because we ensured that there are no small cuts and excluded the embeddings
  with small representativity, every face is a center of a flower.
  Again we distinguish two cases depending on the maximum size of a face $\Delta$:
  If $\Delta \geq \sqrt[3]{f}$, we choose the largest face to be a center
  of a flower obtaining $\max\{2^{\sqrt[3]{f} - 3}, 2\} \geq 2^{\sqrt[3]{2f}/100}$
  CiDCs.
  Otherwise we color the square of the dual
  with $(\sqrt[3]{f})^2$ colors so the largest color class has
  size $f/\sqrt[3]{f}^2 \geq \sqrt[3]{2f}/100$ and we choose it as centers
  of the flowers.
\end{proof}

There also exists an analogue of Theorem~\ref{thm:planar_square_color} for every
fixed surface so for a fixed surface we can improve the exponent:

\begin{theorem}[Amini et at\.~\cite{amini2013}]\label{thm:square_color}
Let $\sigma$ be a fixed surface. Then there exists $c \in \R$ such that
the chromatic number of the square of a graph $G$ with embedding into $\sigma$
with maximum degree $\Delta$
is $\chi(G^2) \leq c\Delta$.
\end{theorem}

\begin{corollary}\label{cor:surface-lb}
  Let $\sigma$ be a fixed surface.
  Every bridgeless cubic graph with embedding in the surface $\sigma$
  and with representativity at least 4 has $2^{\Omega(\sqrt{n})}$
  circuit double covers.
\end{corollary}

This theorem has two potential weaknesses: It might happen that the Euler
characteristic of the embeddings will decrease too fast leading to a constant
number of faces. We are currently not aware of any graph sequence for which
this happens for all the possible embeddings.

The second weakness is that all the embeddings might have small representativity.
While working on Grünbaum's conjecture, Mohar and Vodopivec~\cite{mohar2006polyhedral} 
proved that this is the case for Flower snarks (defined by Isaacs~\cite{flower-snarks},
for small examples see Figure~\ref{fig:flower-snarks}):

\begin{figure}[thb]
  \centering
  \begin{subfigure}[c]{0.3\textwidth}
    \includegraphics[width=\textwidth]{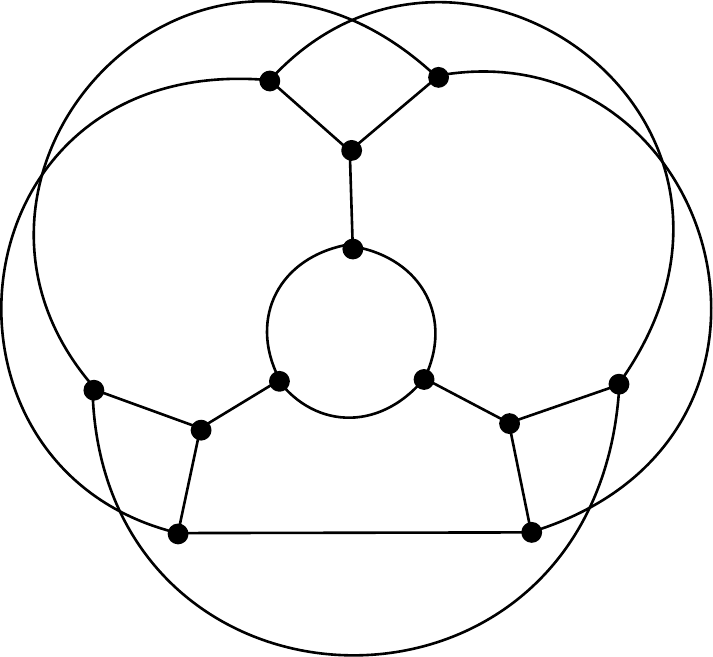}
  \end{subfigure}\hfil
  \begin{subfigure}[c]{0.3\textwidth}
    \includegraphics[width=\textwidth]{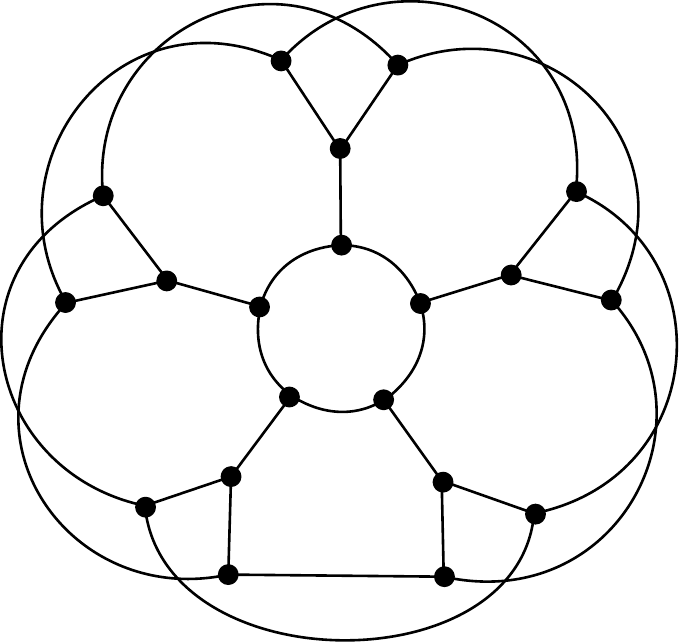}
  \end{subfigure}\hfil
  \begin{subfigure}[c]{0.3\textwidth}
    \includegraphics[width=\textwidth]{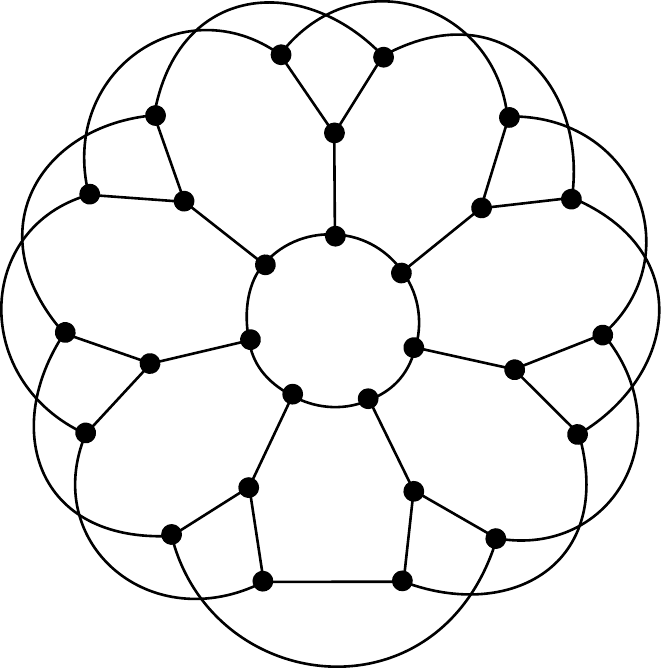}
  \end{subfigure}
  \caption{Flower snarks $J_3$, $J_5$ and $J_7$}
  \label{fig:flower-snarks}
\end{figure}

\begin{theorem}[Mohar and Vodopivec~\cite{mohar2006polyhedral}]
  All embeddings of a Flower snark~$J_k$ for $k \geq 5$ have representativity at most 2.
\end{theorem}

It is therefore natural to ask how many CiDCs do Flower snarks have. We answer
this question in our paper~\cite{lin-rep} where we prove that
Flower snark with $n$ vertices has $\Theta(2^n)$ circuit double covers.

Also the theorems proven above give better bounds
if the graph either has a face of a linear size or has the sizes of the faces
bounded
by a constant. This leads to a question whether we do not loose too much by just
changing the strategy in the middle and whether there is some better way.

We answer this question negatively at least for the planar case.
The counterexample is an {\em antiflower} of size $k$:
\begin{itemize}
  \item Take a flower of size $k$,
  \item add another layer of faces (we call them green faces)
    around the flower so every non-central
    face (the purple faces) of the flower touches $k$ new faces,
  \item contract the outer face into a single vertex, and
  \item expand this vertex into a path so the graph is cubic and
    each of the green faces touches at most 4 other green faces.
\end{itemize}
Figure~\ref{fig:antiflower} shows
an antiflower of size 4, the outer face in the figure is the central face of the
flower used in the construction. Generally an antiflower of size $k$ has
$k$ purple faces and $k^2 - k$  green faces so each
purple face is incident with $k$ green faces.

\begin{figure}
  \centering
  \includegraphics[width=.7\textwidth]{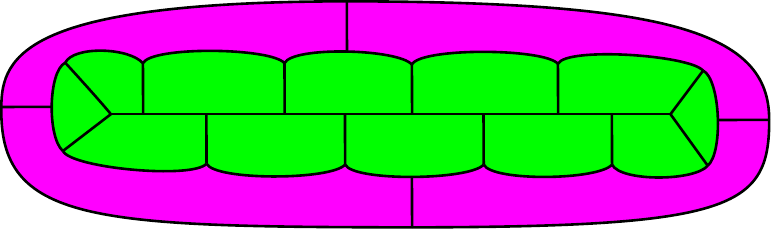}
  \caption{Antiflower of size four}
  \label{fig:antiflower}
\end{figure}

\begin{observation}
  Using the flower construction, no matter how we choose centers of the flowers,
  we obtain at most $2^{O(\sqrt{n})}$ CiDCs for the antiflower with $n$ vertices
  (the size of the antiflower is $\Theta(\sqrt{n})$).
\end{observation}
\begin{proof}
  Note that the antiflowers are 3-edge-connected so the presented embedding
  into the plane is the only possible one (modulo changing which face is the outer one),
  as proven by Whitney~\cite{Whitney}.
  We analyse the possible choices of flowers. If we choose outer face as the center
  of the flower then there is no other flower disjoint with it and it has size
  $O(\sqrt{n})$. If we choose a purple face then the outer face is part of the
  flower so we can choose at most one purple face and the flower has size
  $O(\sqrt{n})$ again.

  Only the green faces remain. But every flower with center in the green face
  has size at most 6 as it neighbors with at most 4 other green faces
  and at most 2 purple faces. And when we select a green face, we cannot select
  another green face which neighbors with the same purple face. So we
  can select at most $\Theta(\sqrt{n})$ green faces at once.
\end{proof}

We did not prove any lower bound specially for antiflowers but they are
planar graphs so the general exponential bound proved below
(Theorem~\ref{thm:planar-exp-lb}) applies. Thus, the observation above describes only 
the limitations of the flower construction. 

\section{The Linear Representation}\label{sec:lin-rep}

In this section we describe what we call a {\em linear representation} for counting circuit double covers. 
In our paper~\cite{lin-rep} we develop this idea further: we define general (non-linear) 
representations, we study their properties and we define a linear representation
for counting CiDCs.
We remind that the number of CiDCs in a graph~$G$ is denoted $\Par(G)$.

Linear representations operate on {\em ordered multipoles}
as defined, e.g., by Nedela and Škoviera \cite{nedela96}. A {\em multipole}
$g = (V(g), E(g))$ is a pair of disjoint finite sets of vertices and, resp., edges.
Each edge $e \in E(g)$ has two ends and each end may or may
not be incident with a vertex.
There are three types of edges with two ends:
\begin{itemize}
  \item {\em Link} is an edge such that both its ends are incident with vertices.
  \item {\em Dangling edge} is an edge with only one end incident with a vertex.
  \item {\em Isolated edge} is an edge with no end incident with a vertex.
\end{itemize}
An end of an edge which is not incident with a vertex is called {\em semiedge}.
We denote $S(g)$ the set of all semiedges of a multipole $g$. A multipole $g$
is an {\em ordered multipole} if it is equipped with a linear order on $S(g)$.
If a multipole $g$ has $|S(g)| = k$, we say that it is a {\em $k$-pole} and
that it has size $k$. (Do not confuse the size with the order of a multipole which
is the number of its vertices.)

Let $g$ be a multipole and $s$, $s'$ its different semiedges. Then we may
create a new multipole $g'$ by joining these semiedges together. Formally
we distinguish two cases:
\begin{itemize}
  \item Both $s$ and $s'$ are ends of the same edge $e$. Then joining them
    together would create a vertex-less loop which we do not allow.
    Hence we create $g'$ by just removing $e$.
  \item Otherwise let $e$ be the edge containing $s$ and $e'$ the edge
    containing $s'$. Let $o$ and $o'$ be the other ends of $e$, resp., $e'$.
    Then $g' = (V(g), (E(g) \setminus \{ e, e' \}) \cup \{ f \})$ where
    $f$ is a new edge with ends $o$, $o'$.
\end{itemize}
{\em Gluing}, denoted $\J_g$, (called junction by Nedela and
Škoviera~\cite{nedela96}) is an operation which takes two ordered
$k$-poles $g$, $g'$ with semiedges $s_1,\dots,s_k$, resp., $s'_1,\dots,s'_k$
according to their orders and performs junction of $s_i$ and $s'_i$ for all
$i = 1, \dots, k$.

We will only work with ordered multipoles so from now on all $k$-poles
are ordered even without explicitly saying so. We allow both $V(g)$
and $E(g)$ to be empty. Note that 0-poles are just graphs.
We say that a multipole is cubic if all its vertices have degree three.
We denote $\kgadgets{k}$ the set of all cubic $k$-poles (as we are only interested
in cubic graphs).

We extend CiDCs to multipoles in the following way:
The CiDC of a multipole is a multiset of circuits and
paths which covers every edge of the multipole twice.
Both ends of each path must be semiedges and no edge or vertex can appear twice in
one path.

Note that given $g = \J_g(g_1, g_2)$, every CiDC of $g$ is uniquely
determined by a CiDC of $g_1$ and of $g_2$ and the way they are joined together.
On the other hand, when we combine a CiDC of $g_1$ and of $g_2$,
we may not get a CiDC. To deal with this will be key part of
the proof of the following fact that will be crucial in the later parts of
the paper.

\begin{theorem}\label{thm:lin-rep-short}
  For any $k \in \N$ there exists $c \in \N$, a function
  $h_\Par : \kgadgets{k} \to \R^c$ and a bilinear function
  $(\J_g)_\Par : \R^c \times \R^c \to \R$ such that for any cubic $k$-poles
  $g_1$, $g_2$ it holds
  $$\Par(\J_g(g_1, g_2)) = (\J_g)_\Par(h_\Par(g_1), h_\Par(g_2))$$
  and $h_\Par(g) \geq 0$ for every cubic $k$-pole $g$.
  We call the value $h_\Par(g)$ a {\em multiplicity vector}.
  Moreover, for $k \in \set{0, 2, 3}$ we may take $c = 1$.
\end{theorem}

\begin{proof}
The theorem would be obvious if we allowed $c = \omega$ -- then
we could take $\R^{\kgadgets{k}}$ and define $h_\Par(g) = e_{g}$.
The hard part is coming up with a finite representation. To do this,
we observe that for each $k$ there is only finitely many ways how
a CiDC can behave on the semiedges. % (we define CiDCs of a multipole below). 
We call these ways {\em boundaries} and we denote the set of them
$B^k = \{ B^k_i : i = 1, 2, 3, \dots, |B^k| \}$. Then:
\begin{itemize}
  \item we let $c = |B^k|$ be the number of boundaries; 
  \item every coordinate of the vector $h_\Par(g) \in \R^c$ is the number
    of CiDCs of $g$ with this boundary, 
      scaled down by $(1/2)^f$ where $f$ is the number of isolated edges in $g$;
      %as we will see below.
      and
  \item we represent bilinear mapping $(\J_g)_\Par$ by a matrix in~$\R^{c \times c}$
    so that the element $((\J_g)_\Par)_{i,j}$ of the matrix is the number of CiDCs created by
    joining one CiDC with boundary $B^k_i$ with one CiDC with
    boundary $B^k_j$. This number can be at most $2^k$ (number of ways to 
    identify the two paths on one semiedge with the two paths at the 
    other semiedge we will identify with it). But it is often smaller, 
    as many of the identifications do not lead to a collection of circuits.
\end{itemize}
Below we describe how to specify these boundaries. Our approach leads
to $2^{\Theta(k^2)}$ boundaries of size $k$. (We prove in~\cite{lin-rep}
that no linear representation counting CiDCs can have less than
$2^{\Omega(k \log k)}$ boundaries of size $k$.)

When we glue two $k$-poles, we get $k$ binary choices
how to join the paths on the newly created edges. The question is what
we need to know about CiDCs inside the multipole to determine whether a particular
choice leads to a valid CiDC or not.
Because we are interested in cubic graphs, we can assume that no two paths
in a CiDC are identical (we ignore isolated edges for now).

To simplify the reasoning we proceed joining two semiedges at time.
So to be able to join the two semiedges we need to remember for every path 
which semiedges are its end points and which other paths it shares an edge
with. We will record this in the following way: The paths will be labeled
consecutively, starting at one. For each semiedge we record the labels of the
two paths incident with it, \eg the only possible configuration on boundary
of a gadget of size~$3$ 
has the first semiedge with paths 1 and 2, the second one with paths 1 and 3
and the last semiedge with paths 2 and 3.
We write this boundary $\Boundary{(1,2), (1,3), (2,3)}{}$.

Then for every two paths that share an edge and
it is not yet known from the description of the semiedges, we record a tuple
containing the numbers of these two paths. We write this after $|$ sign,
\eg if paths 5 and 6 are incident only inside the multipole, we would
write $\Boundary{\dots}{(6,5)}$.

The same boundary can obviously be written in many ways -- the operations
that preserve the same structure are renaming the paths,
swapping the order of the numbers in each tuple, and permuting
the tuples describing incidences inside the multipole.
To get rid of this non-uniqueness we just take the equivalence classes
under all these operations. In the implementation we represent
each class by its lexicographically minimal element.
Note that:
\begin{itemize}
  \item The 0-poles -- \ie graphs -- have exactly one boundary
    written $\Boundary{}{}$.
  \item There is no boundary of size one because a graph with a bridge cannot
    have a CiDC.
  \item There is one boundary of size two written $\Boundary{(1,2), (1,2)}{}$.
  \item There is one boundary of size three, namely
    $\Boundary{(1,2), (1,3), (2,3)}{}$.
\end{itemize}

Thus we verified the ``moreover'' part of the theorem.
We prove in~\cite{lin-rep} that the described representation
has 33 boundaries of size four. We also prove that there exists
a better representation with only 21 boundaries.

The multiplicity vector $h_\Par(g)$ for multipole $g$ describes
how many CiDCs with each boundary $g$ has. The value $(h_\Par(g))_i$ is $x / 2^f$ 
where $x$ is the number of CiDCs of $g$ with the $i$-th boundary
and $f$ is the number of isolated edges in $g$.
Why is the scaling by $f$ needed?

We have defined $((\J_g)_\Par)_{i,j}$ as the number of choices 
(from 0 to $2^k$) that join a CiDC with the $i$-th boundary with a CiDC with
$j$-th boundary and the result is a CiDC. This number is determined just by the
two boundaries, 
that is it does not depend on the particular multipoles we are joining. 
However, based on the choice of multipoles, some of the choices counted 
by~$((\J_g)_\Par)_{i,j}$ give the same set of circuits. 
This happens if (and only if) the choice between two paths ending in one semiedge is 
irrelevant, as the two paths share all edges.

However, we are dealing with cubic 
multipoles, so there is only one way how this can happen. 
Each isolated edge in any of the multipoles causes two 
of the choices to produce the same CiDC: for any isolated edge~$e$ 
changing the way we glue the CiDCs of~$g_1$ and of~$g_2$ at the both ends of~$e$ 
(which are different two of the $k$ semiedges) leaves us with the same 
set of circuits. To compensate for this we need to scale the value $h_\Par(g)$ by
$(1/2)^f$ where $f$ is the number of isolated edges in the multipole $g$.

Then when we join an isolated edge with anything, each CiDC is counted twice
so it cancels out the coefficient $1/2$ of the disappearing isolated edge.
\end{proof}

An easy exercise is to apply this theorem to replacement of a cubic vertex
with a triangle:

\begin{observation}\label{obs:triange_to_vertex}
  Replacing a vertex with a triangle in a cubic graphs doubles the number of circuit
  double covers.
\end{observation}
\begin{proof}
  The dimension of the multiplicity vector of a 3-vertex is one
  (Theorem~\ref{thm:lin-rep-short}).
  The graph $K_4$ is obtained from three parallel edges (denoted $K_2^3$)
  by replacing a vertex by a triangle, and $\Par(K_4) = 2$ while $\Par(K_2^3) = 1$.
  Hence the multiplicity vector of triangle is twice the multiplicity vector
  of a 3-vertex.
  (See Figures~\ref{fig:cdc-3-vertex} and~\ref{fig:cdc-triangle}.)
\end{proof}

Although a very simple observation, it gives us an infinite class of graphs
with exactly $2^{n/2 - 1}$ CiDCs. Note that these graphs are called Klee graphs. 
They are planar, in fact they are exactly the uniquely edge-3-colorable cubic
planar graphs~\cite{fowler-phd}. 

\begin{corollary}\label{cor:triangle}
  An $n$-vertex graph created from three parallel edges by repeatedly expanding
  vertices to triangles has exactly $2^{n/2 - 1}$ circuit double covers.
\end{corollary}

Note that Corollary~\ref{cor:triangle} shows that Conjecture~\ref{con:cdc-count}
cannot be any
stronger as there is an infinite family of graphs for which this conjecture is
tight. On the other hand a stronger version might hold for triangle-free
or more cyclically connected graphs.
We conclude this section with a result
we obtained in \cite{lin-rep} where we tested all $\set{C_3, C_4}$-free
cubic biconnected graphs on 20 or less vertices.\footnote{
  We excluded $C_3$ and $C_4$ due to Corollary~\ref{cor:fin} below.
}
We discovered the graph closest to the bound of Conjecture~\ref{con:cdc-count}
is Petersen graph with ratio $3.25$. Two more tested graphs had the ratio
$\nu(G)/2^{\abs{V(G)}/2 - 1}$ less than 10, all other had higher ratios.
For details see Figure~\ref{fig:cdc-plot}. The blue points represent tested
graphs and the purple line is the lower bound of Conjecture~\ref{con:cdc-count}.
The data were obtained by experiment\footnote{
  The experiments together with other code can be found in our gitlab repo
  \url{https://gitlab.kam.mff.cuni.cz/radek/cdc-counting}.
} \Experiment{test_exp_cdc.sh}.

\begin{figure}[tbh]
  \centering
  \begin{tikzpicture}
    \begin{axis}[
        width=0.85\textwidth,
        height=8cm,
        ymode=log,
        xmin=9,
        xmax=21,
        xtickmin=10,
        xtickmax=20,
        xlabel={The number of vertices},
        ylabel={The number of CiDCs (log scale)},
        legend style={ legend pos=north west },
      ]
      \addplot [blue, only marks] table [header=false, col sep=comma] {img/num-cdc-data.csv};
      \addplot [purple, domain=9.5:20.5] {2^(x/2 - 1)};
      \legend {Tested graphs, Conjecture~\ref{con:cdc-count}};
    \end{axis}
  \end{tikzpicture}
  \caption{The number of CiDCs of $\set{C_3, C_4}$-free cubic biconnected graphs}
  \label{fig:cdc-plot}
\end{figure}
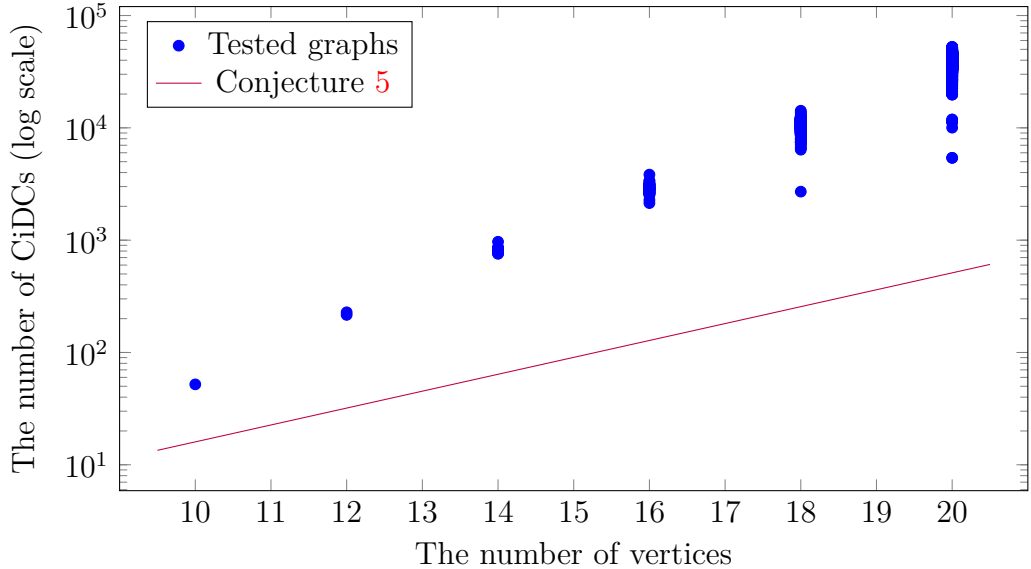

\section{Reducing Cycles}

In this section we combine the framework with linear programming
to obtain a better bound on the number of CiDCs of planar graphs.
First we describe the method in general and then we apply it to
the number of CiDCs of planar graphs. It is also straightforward to use
this method for other linear representations.

\subsection{General Method}

We want to show that graphs in some class $\mathcal C$ have many CiDCs and
we know that there is a small set of multipoles $\mathcal S$ such that every graph
in $\mathcal C$ either has some multipole of $\mathcal S$ as an induced subgraph
or it is trivial in some sense. We denote the class of the trivial cases
$\mathcal B \subset \mathcal C$. The usual reasons for a graph to be considered
trivial are a small number of vertices and existence of small cuts.

We can for every multipole $s \in \mathcal S$ choose a set of multipoles
with fewer vertices
$\mathcal R_s$ and try to prove that the number of CiDCs of a graph $G$
containing $s$ can be bounded from below by the number of CiDCs of $G$ with $s$
replaced by elements of $\mathcal R_s$ and that these graphs also belong
to $\mathcal C$. This allows us to proceed by induction on the number of
vertices.

A bit more formally (for an application see the proof of
Theorem~\ref{thm:planar-exp-lb}):
Suppose we are proving a lower bound of a form $c^{n(G) - d}$
where $c > 1$, $d \in \R$. Then we want to show the inequality
$$\Par(\J_g(g, s)) \geq
    \min_{r \in \mathcal R_s} c^{n(s) - n(r)} \Par(\J_g(g, r))$$
where $n(g)$ is the number of vertices\footnote{
  Do not confuse with $|g|$ which denotes the size of the multipole, \ie
  the number of its semiedges.
} of a multipole $g$.
Suppose this is true for all $\J_g(g, s) \in \mathcal C$, all the graphs
$\J_g(g, r)$ also belong to $\mathcal C$ and we have other means to prove
the bound for graphs in $\mathcal B$. Then we can prove the desired
lower bound $c^{n(G) - d}$ for every $G \in \mathcal C$ by induction
on the number of vertices
using this formula as the induction step for the non-trivial graphs.

Proving this formula for each $s$ is where the linear programming
comes into play. We saw a special case of this approach before in
Observation~\ref{obs:triange_to_vertex}. But in that case there was
only one boundary and one substitution 3-pole, so no linear program
was needed.
The idea of the program is similar to the LP relaxation of integer programs, 
except now the ``real'' program is restricted not to integer vectors 
but to representations of multipoles. 

%fix the value of the right-hand side and search for as small left-hand side
%as possible. Because we cannot search all gadgets, we search all possible
%multiplicity vectors instead.

\begin{theorem}\label{thm:lp}
  Let $c > 1$, $s$, $\mathcal R_s$ be defined as above
  and to simplify the notation put $\J = \J_g$.
  If the objective value of the linear program $P$ (described below)
  is at least~1 then
  the following holds for all multipoles $g$:
  \begin{equation}\label{eq:wanted}
    \Par(\J(g, s)) \geq
      \min_{r \in \mathcal R_s} c^{n(s) - n(r)} \Par(\J(g, r))
  \end{equation}
  where $n(g)$ is the number of vertices of multipole $g$ and
  the linear program $P$ is:
  \begin{align*}
    \min_{m \in \R^{B^{|s|}}} \J_\Par(m, h_\Par(s))& \\
    c^{n(s) - n(r)} \J_\Par(m, h_\Par(r)) &\geq 1 \quad \forall r \in \mathcal R_s\\
    m &\geq 0 
  \end{align*}
\end{theorem}

\begin{proof}
  Consider a particular multipole~$g$. 
  Recall (Theorem~\ref{thm:lin-rep-short} and its proof) that $\J_\Par$ is a bilinear function and $h_\Par(g) \ge 0$. 
  To show~\eqref{eq:wanted}, let $\alpha$ be the right-hand side of this inequality. 
  This means that 
      $$
        \alpha = \min_{r \in \mathcal R_s} c^{n(s) - n(r)} \J_\Par(h_\Par(g), h_\Par(r)). 
      $$
  If $\alpha = 0$ then~\eqref{eq:wanted} is trivially true. Otherwise, put $m = h_\Par(g)/\alpha$. 
  Due to the bilinearity of~$J_\Par$, vector $m$ is a feasible solution to~$P$. 
  If the objective value of~$P$ is at least~1, then (again by bilinearity) we have in particular
  $\J_\Par(h_\Par(g)/\alpha, h_\Par(s)) \ge 1$. This implies $\J_\Par(h_\Par(g), h_\Par(s)) \ge \alpha$ which
  is equivalent to~\eqref{eq:wanted}. 
\end{proof}

Note that if the objective value of the linear program is less than one
but still more than zero then exponential bound for a smaller
$c$ might hold. This theorem also holds for any other
linear representation which maps multipoles to non-negative vectors.

The downside of a linear program is that it is usually solved by a numerical
method which is not suitable for a theoretical proof. We circumvent this
by solving the dual problem. Note that any solution of the dual
gives us a lower bound but of course suboptimal solutions will give
weaker bounds. So for a given solution of the dual we only need to certify
that it is indeed a solution, that is, to verify that all the required
inequalities hold (we do not need to prove optimality of the solution).

\subsection{Application to Planar Graphs}

We are interested in bridgeless cubic planar
graphs. We know that every such simple graph contains a cycle of size at most five
because its dual is also a planar graph and so it contains a vertex of degree
at most five (due to Euler's formula). (And multigraph contains a small cut.)
Moreover, for $c \leq \sqrt{2}$ and $d \leq 2$
we may reduce the cuts of size two and three
due to Observations~\ref{obs:cdc_2cut} and~\ref{obs:cdc_3cut}.
So we take all bridgeless planar cubic graphs
as the class $\mathcal C$ and we define $\mathcal B$ to be
all graphs in $\mathcal C$ which are not cyclically 4-edge-connected.

We need to be able to replace 4-cycles and 5-cycles 
(cycles of the length 4, or 5, resp.). The important
observation is that if the graph is 3-edge-connected then all the 4-cycles
and 5-cycles are faces. What can we replace them with?
We need the replacements to be smaller. 
Ignoring the labeling of the semiedges we have the following options:
either a tree on two vertices or two isolated edges for the 4-cycle
(see Figure~\ref{fig:repl-4}) and
a tree on three vertices or
a combination of a cubic vertex and an isolated edge for the 5-cycle
(Figure~\ref{fig:repl-5}).

\begin{figure}
  \centering
  \includegraphics[scale=.7]{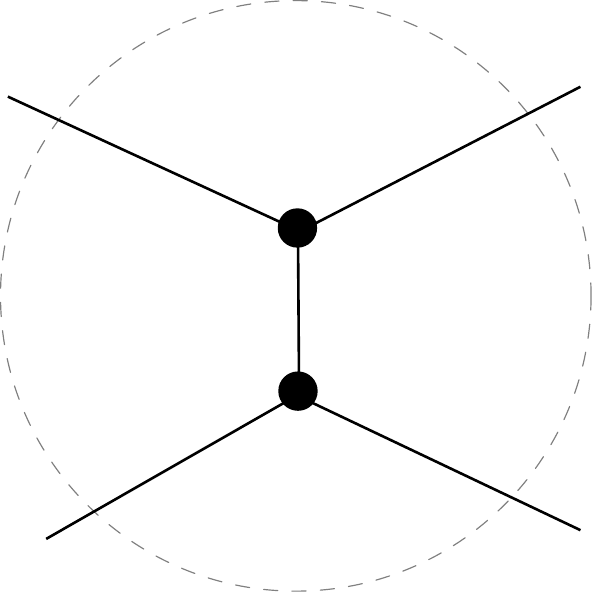}\hfil
  \includegraphics[scale=.7]{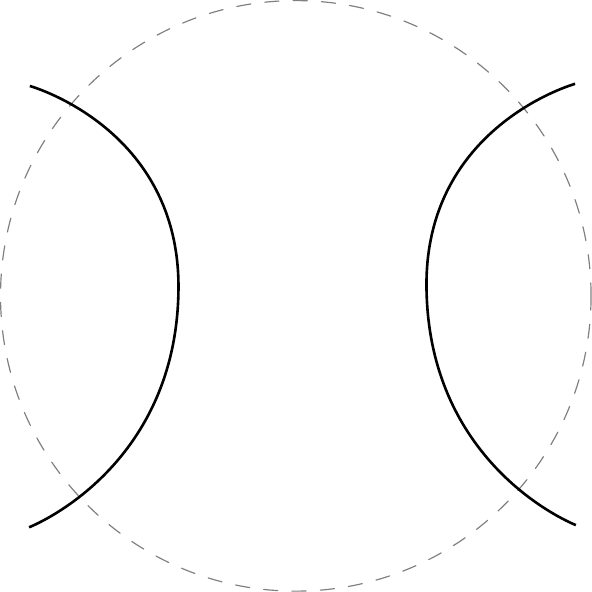}
  \caption{Possible replacements for a 4-cycle (up to permutation of semiedges).%
    \label{fig:repl-4}}
\end{figure}

\begin{figure}
  \centering
  \includegraphics[scale=.7]{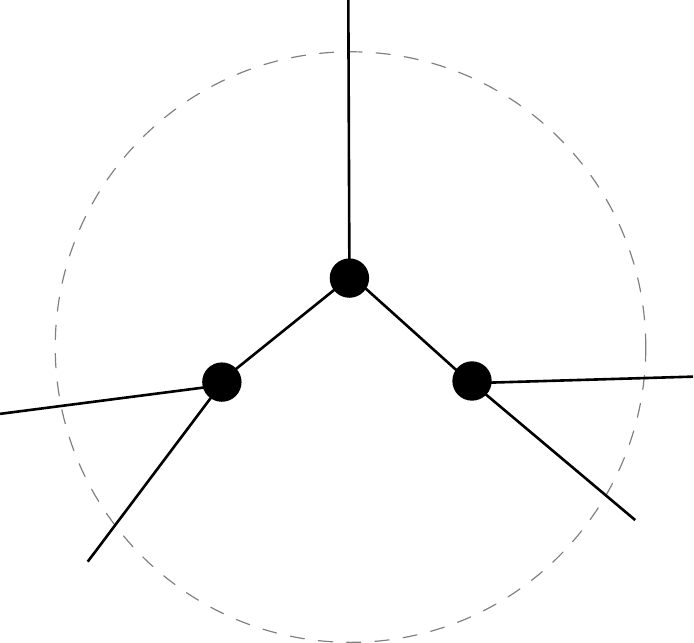}\hfil
  \includegraphics[scale=.7]{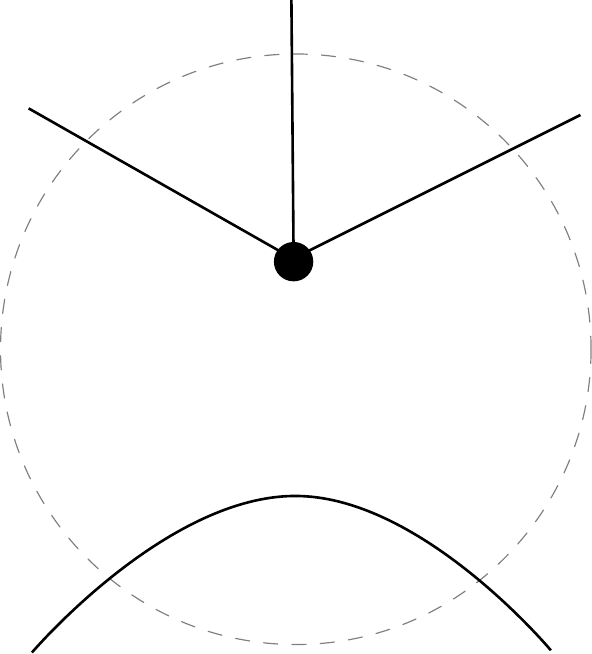}
  \caption{Possible replacements for a 5-cycle (up to permutation of semiedges).%
    \label{fig:repl-5}}
\end{figure}

We tried all the suggested replacements and the ones using trees never
gave a stronger bound than the ones using only free edges and cubic vertices.
Hence we will replace the 4-cycles with the two possible
non-crossing choices of two isolated edges and the 5-cycles with a cubic vertex
and an isolated edge (again drawn in a non-crossing way) in all the 5 possible rotations.
The following theorems show the results of this replacements:

\begin{theorem}\label{thm:4cyc}
  Let $G$ be a cyclically 4-edge-connected cubic graph with a 4-cycle.
  Let $G_1$ and $G_2$ be the two possible graphs obtained from $G$ by deleting
  two opposite edges of the 4-cycle and suppressing vertices of degree 2.
  Then $\Par(G) \geq 4 \min\set{\Par(G_1), \Par(G_2)}$.
\end{theorem}
\begin{proof}
  Because the graph is cyclically 4-edge-connected and we are deleting
  non-adjacent edges, the resulting graph is still 2-edge-connected
  due to Observation~\ref{obs:cyc-rem2}.
  We apply Theorem~\ref{thm:lp} with $c = \sqrt{2}$, 4-cycle as $s$ and
  the two non-crossing choices of two isolated edges $\mathcal R_s$. We obtain the
  following linear program:
  \begin{align*}
    \max \sum_{i=0}^{32} o_i m_i& \\
    \left(\sqrt{2}\right)^{4} \sum_{i=0}^{32} a_i m_i &\geq 1\\
    \left(\sqrt{2}\right)^{4} \sum_{i=0}^{32} b_i m_i &\geq 1
  \end{align*}
  where $m_i$ are the variables and $o_i$, $a_i$ and $b_i$ are constants
  computed by experiment \Experiment{reduce-cycle.py}
  (although they might be computed by hand in
  this case). Each variable corresponds to a boundary and
  there is 33 boundaries of the size four hence
  there is 33 variables. Each inequality corresponds to an element of
  $\mathcal R_s$. Plugging in the values, taking dual and removing
  31 conditions obviously implied by other conditions, we get:
  \begin{align*}
    \max \frac{1}{4}x_0 + \frac{1}{4}x_1& \\
    x_0 &\leq 2\\
    x_1 &\leq 2 \\
    x_i &\geq 0 \quad \forall i \in \set{0, 1}
  \end{align*}
  The objective value of this linear program is 1. This satisfies the conditions of the
  theorem so we obtain:
  \begin{gather*}
    \Par(G) = \Par(\J(g, s)) \geq
    \min_{r \in \mathcal R_s} \left(\sqrt{2}\right)^{4} \Par(\J(g, r)) =
    4 \min\set{\Par(G_1), \Par(G_2)}. \qedhere
  \end{gather*}
\end{proof}

\begin{theorem}\label{thm:5cyc}
  Let $G$ be a cyclically 4-edge-connected cubic graph with a 5-cycle
  and no 4-cycle.
  Let $G_1, G_2, \dots, G_5$ be the 5 possible graphs obtained from $G$ by
  replacing the 5-cycle by a cubic vertex and an edge in non-crossing way
  (assuming the 5-cycle is a face).
  Then $\Par(G) \geq 5/2 \min_i \Par(G_i)$. If we replace the 5-cycle by
  a cubic vertex and an edge in all possible ways (\ie breaking planarity)
  we get $\Par(G) \geq 3.75 \min_i \Par(G_i)$.
\end{theorem}
\begin{proof}
  We can simulate this replacement by removal of the two edges adjacent with
  the 5-cycle and the two vertices which should be connected by an edge in
  the result, contracting 5-cycle into a vertex and adding an edge to
  join the two vertices of degree two. These two vertices are distinct
  otherwise there would be a 4-cycle in the graph. So the deleted
  edges were not adjacent and the resulting graph is 2-edge-connected
  due to Observation~\ref{obs:cyc-rem2}.

  The rest of the proof is analogous to the proof of the previous theorem
  and the computer aided part is also a part of experiment
  \Experiment{reduce-cycle.py}.
  The (simplified) dual program for non-crossing case with $c = \sqrt[4]{5/2}$ is:
  \begin{align*}
     \max 0.4x_0 + 0.4x_1 + 0.4x_2 + 0.4x_3 + 0.4x_4& \\
    x_i &\leq 1 \quad \forall i \in \set{0, 1, 2, 3, 4} \\
    x_2 + x_4 &\leq 1 \\
    x_1 + x_4 &\leq 1 \\
    x_1 + x_3 &\leq 1 \\
    x_0 + x_3 &\leq 1 \\
    x_0 + x_2 &\leq 1 \\
    x_i &\geq 0 \quad \forall i \in \set{0, 1, 2, 3, 4} \\
  \end{align*}
  The solutions is $x_i = 0.5$ and the objective value is 1.
  The crossing case with $c = \sqrt[4]{3.75}$ and all the 10 possible
  replacement 5-poles (5 non-crossing plus 5 crossing) leads to
  a larger linear program so we omit it. But again the objective
  value is one, so we can apply the Theorem~\ref{thm:lp}.
\end{proof}

In both theorems the modified graphs have four vertices less than the original
ones so the best we can prove is that the number of CiDCs increases by the factor
2.5 (due to 5-cycles)
with addition of four vertices.
So we obtain:

\begin{theorem}\label{thm:planar-exp-lb}
  Every bridgeless cubic planar (multi)graph
  has at least $(5/2)^{n/4 - 1/2}$ circuit double covers.
\end{theorem}
\begin{proof}
  By induction on $n$, the number of vertices of the graph $G$.
  The base cases are three parallel edges and $K_4$
  because these are the only planar cubic graphs without a non-trivial cut.
  The three parallel edges graph has only one CiDC which exactly matches
  the bound. The $K_4$ has two CiDCs and the bound requires only approximately
  $1.58$.

  Suppose $n \geq 6$. If $G$ has non-trivial cut
  of size two, we apply Observation~\ref{obs:cdc_2cut} and we obtain
  \begin{align*}
    \Par(G) &\geq 2\Par(G_1)\Par(G_2)
    \geq 2 (5/2)^{|V(G_1)|/4 - 1/2} (5/2)^{|V(G_2)|/4 - 1/2} \\
    & = 2  (5/2)^{n/4 - 1/2} \sqrt{2/5} \geq (5/2)^{n/4 - 1/2}
  \end{align*}
  which we needed.
  Similarly for a non-trivial 3-cut we use Observation~\ref{obs:cdc_3cut}:
  \begin{align*}
    \Par(G) &\geq \Par(G_1)\Par(G_2)
    \geq (5/2)^{|V(G_1)|/4 - 1/2} (5/2)^{|V(G_2)|/4 - 1/2} \\
    & =  (5/2)^{(n + 2)/4 - 1}  = (5/2)^{n/4 - 1/2}.
  \end{align*}

  So $G$ is cyclically 4-edge-connected. It is 3-edge-connected
  so each facial walk in its planar embedding is a circuit.
  Due to Euler's formula, the planar dual of $G$ must have a vertex of
  degree at most five. Hence $G$ has a face of size at most five.
  We already excluded 3-faces because the cut around
  a triangle is a non-trivial cut of size three. If $G$ has a 4-face,
  we apply Theorem~\ref{thm:4cyc}:
  \begin{gather*}
    \Par(G) \geq 4 (5/2)^{(n - 4)/4 - 1/2} =
    \frac{4 \cdot 2}{5} (5/2)^{n/4 - 1/2} \geq (5/2)^{n/4 - 1/2}.
  \end{gather*}
  Otherwise $G$ has a 5-face and we apply Theorem~\ref{thm:5cyc}:
  \begin{gather*}
    \Par(G) \geq (5/2) (5/2)^{(n - 4)/4 - 1/2} = (5/2)^{n/4 - 1/2}.
    \qedhere
  \end{gather*}
\end{proof}

To compare this with Conjecture~\ref{con:cdc-count}, $(5/2)^{n/4} = 2^{cn}$
for $c$ approximately $0.33$ so this is still a weaker
bound than Conjecture~\ref{con:cdc-count} asks for.
We also tried to apply Theorem~\ref{thm:lp} to 6-cycles but we did not have
found a suitable set $\mathcal R$ -- neither isolated edges nor trees worked.
We conclude this section with a summary of what we know about a hypothetical
counterexample to Conjecture~\ref{con:cdc-count}:

\begin{corollary}\label{cor:fin}
  A minimal counterexample (the one with the smallest number of vertices) to
  Conjecture~\ref{con:cdc-count}:
  \begin{enumerate}
    \item does not have a 2-edge-cut,
    \item does not have a non-trivial 3-edge-cut,
    \item does not contain a triangle,
    \item does not contain a 4-cycle, and
    \item has at least 22 vertices.
  \end{enumerate}
\end{corollary}

The first three points are due to Observations~\ref{obs:cdc_2cut}
and~\ref{obs:cdc_3cut}, the fourth one due to Theorem~\ref{thm:4cyc}
and the last one was verified by a computation on all 3-edge-connected
cubic graphs up to 20 vertices (see \Experiment{test_exp_cdc.sh} experiment).
Note that we cannot exclude 5-cycles
as the bound provided by the second part of Theorem~\ref{thm:5cyc} is too weak.

\section*{Acknowledgements}

Large part the work was done during first author's Ph.D. study at
Computer Science Institute of Charles University.
The first author was supported by Institutional support
of research organization development (RVO18000)
of The Ministry of Education, Youth and Sports of Czech Republic.
The second author was supported by grant 22-17398S of the Czech Science Foundation. 

We thank to the anonymous referees for their helpful comments. 

% Institutional support of research organization development - ID: 40683 - National/external code: RVO18000 - Researcher: doc. Ing. Štěpán Starosta,

\bibliographystyle{amsplain}
\bibliography{bibliography}

\listoftodos

\end{document}